\documentclass[times, 11pt, doublespace]{amsart}
\usepackage{amsmath, amssymb, amsthm,verbatim}

\numberwithin{equation}{section}

\usepackage{setspace}
\usepackage[parfill]{parskip}

\usepackage[margin=1.5 in]{geometry}
\usepackage{tikz-cd}

\theoremstyle{plain}
\newtheorem{theorem}[equation]{Theorem}
\newtheorem{proposition}[equation]{Proposition}
\newtheorem{lemma}[equation]{Lemma}
\newtheorem{cor}[equation]{Corollary}
\newtheorem{conj}[equation]{Conjecture}

\newcounter{intro}
\newtheorem{introtheorem}[intro]{Theorem}
\newtheorem{introcor}[intro]{Corollary}

\theoremstyle{definition}
\newtheorem{definition}[equation]{Definition}

\newtheorem{remark}[equation]{Remark}

\DeclareMathOperator{\Spec}{Spec}
\DeclareMathOperator{\Pic}{Pic}

\DeclareMathOperator{\NS}{NS}
\DeclareMathOperator{\pr}{pr}
\DeclareMathOperator{\Jac}{Jac}
\DeclareMathOperator{\Prep}{Prep}

\DeclareMathOperator{\Alb}{Alb}

\DeclareMathOperator{\NT}{NT}
\DeclareMathOperator{\an}{an}
\DeclareMathOperator{\Div}{Div}
\DeclareMathOperator{\Prin}{Pr}
\DeclareMathOperator{\Mod}{mod}
\DeclareMathOperator{\Int}{int}
\DeclareMathOperator{\Cont}{cont}
\DeclareMathOperator{\Nef}{nef}
\DeclareMathOperator{\Vrt}{vert}
\DeclareMathOperator{\Red}{red}
\DeclareMathOperator{\Id}{id}
\DeclareMathOperator{\Nt}{NT}
\DeclareMathOperator{\Tor}{tor}
\DeclareMathOperator{\Sup}{sup}
\DeclareMathOperator{\tr}{tr}
\DeclareMathOperator{\Ber}{Ber}

\renewcommand{\Im}{\operatorname{Im}}
\renewcommand{\div}{\text{div}}
\newcommand{\Tr}{\operatorname{Tr}}
\newcommand{\isom}{\xrightarrow{~\sim~}}
\newcommand{\hooklongrightarrow}{\lhook\joinrel\longrightarrow}

\newcommand{\trP}[1]{\operatorname{Pic}_{\tr}(#1)}
\newcommand{\trPo}[1]{\operatorname{Pic}^0_{\tr}(#1)}

\newcommand{\C}{\mathbb{C}}
\newcommand{\R}{\mathbb{R}}
\newcommand{\Z}{\mathbb{Z}} 

\newcommand{\K}{\mathbb{K}}
\newcommand{\Q}{\mathbb{Q}}
\newcommand{\F}{\mathbb{F}}

\newcommand{\ocA}{\overline{\mathcal{A}}}
\newcommand{\cB}{\mathcal{B}}
\newcommand{\cC}{\mathcal{C}}
\newcommand{\cD}{\mathcal{D}}
\newcommand{\ocD}{\overline{\mathcal{D}}}
\newcommand{\cE}{\mathcal{E}}
\newcommand{\cH}{\mathcal{H}}
\newcommand{\ocH}{\overline{\mathcal{H}}}
\newcommand{\cO}{\mathcal{O}}
\newcommand{\cX}{\mathcal{X}}
\newcommand{\cU}{\mathcal{U}}
\newcommand{\cV}{\mathcal{V}}
\newcommand{\cW}{\mathcal{W}}
\newcommand{\cY}{\mathcal{Y}}
\newcommand{\cM}{\mathcal{M}}
\newcommand{\oM}{\overline{M}}
\newcommand{\ocM}{\overline{\mathcal{M}}}

\newcommand{\oN}{\overline{N}}
\newcommand{\calL}{\mathcal{L}}
\newcommand{\oL}{\overline L}
\newcommand{\ocL}{\overline{\mathcal{L}}}
\newcommand{\oH}{\overline H}

\newcommand{\oP}{\overline{P}}

\renewcommand{\P}{\mathbb{P}}

\begin{document}

\title{Heights and arithmetic dynamics over finitely generated fields}

\author{Alexander Carney}

\date{\today}

\maketitle
\begin{abstract}
We develop a theory of vector-valued heights and intersections defined relative to finitely generated extensions $K/k$. These generalize both number field and geometric heights. When $k$ is $\Q$ or $\F_p$, or when a non-isotriviality condition holds, we obtain Northcott-type results. We then prove a version of the Hodge Index Theorem for vector-valued intersections, and use it to prove a rigidity theorem for polarized dynamical systems over any field. 
\end{abstract}

\section{Introduction}\label{introsec}

Arithmetic intersection theory, originally created by Arakelov and developed further by Gillet, Soul\'e, Faltings, Deligne, Szpiro, Zhang, and others, and is often most useful for its ability to produce height functions. Many well-known results in number theory, such as the Mordell-Weil Theorem~\cite{SilvermanECs}, Faltings' Theorem~\cite{faltingsmordell} (as well as Parshin's proof of the Mordell Conjecture over function fields~\cite{Parshin}), and the Manin-Mumford and Bogomolov Conjectures~\cite{raynaudthm1,ullmobogomolov,zhangbogomolov} are obtained using finiteness theorems for their respective height functions.

Such finiteness theorems, or Northcott properties, also provide a connection between arithmetic geometry and dynamics. Let $X$ be a projective variety, and let $f:X\to X$ be an endomorphism. $(X,f)$ is called a \emph{polarizable dynamical system} provided there exists an ample line bundle $L\in\Pic(X)$ such that $f^*L=L^{\otimes q}$ for some $q>1$. The set of preperiodic points of $f$ is the set
\[
\Prep(f):=\left\{x\in X(\overline\K)\big|f^n(x)=f^m(x)\text{ for some }n>m>0\right\}.
\]

$\Prep(f)$ is contained in the set of $f$-canonical height zero points, and the two sets are equal provided Northcott holds~\cite{silvermanbook}. By constructing heights with the Northcott property over all finitely generated fields, we prove the following.

\begin{introtheorem}\label{dynamicstheorem}
Let $X$ be a projective variety over any field $\K$, and let $f$ and $g$ be two polarizable algebraic dynamical systems on $X$. The following four statements are equivalent:
\begin{enumerate}
\item 
$g(\Prep(f))\subset\Prep(f)$.
\item
$\Prep(f)\subset\Prep(g)$.
\item
$\Prep(f)\cap\Prep(g)$ is Zariski dense in $X$.
\item
$\Prep(f)=\Prep(g)$.
\end{enumerate}
\end{introtheorem}

This is proven by Yuan-Zhang~\cite{yz,yz2} when $\K$ has characteristic zero and by the author~\cite{carney} when $\K$ is a transcendence degree one function field. We additionally prove a corollary for local fields.

\begin{introcor}\label{introequilibrium}
Suppose $\K$ is algebraically closed and complete with respect to some absolute value (for example $\C$ or $\C_p$ in characteristic zero). Let $f,g:X\to X$ be two polarizable dynamical systems over $\K$. If $\Prep(f)\cap\Prep(g)$ is Zariski dense in $X$, we have an equality of equilibrium measures, $d\mu_f=d\mu_g$. 
\end{introcor}

Moriwaki~\cite{moriwakifg1,moriwakifg2} develops a theory of heights and proves a Northcott property over finitely generated extensions $K$ of $\Q$ with transcendence degree $d$ by first fixing a polarization: a normal, projective arithmetic model $\cB\to\Spec\Z$ with function field $K$ and a collection of nef $C^{\infty}$-Hermitian line bundles $\ocH_1,\dots,\ocH_d$ on $\cB$. Since every $\C$-variety can be defined over a field finitely generated over $\Q$, this allows him to recover the Manin-Mumford conjecture (Raynaud's Theorem, ~\cite{raynaudthm1,raynaudthm2}) over $\C$.

Because it requires a fixed polarization, however, Moriwaki's method is too weak to extend the Hodge Index Theorem of~\cite{yz} and~\cite{carney} to all finitely generated fields. The Hodge Index Theorem, here Theorem~\ref{hodgeindex}, is the key step in proving an equivalence of canonical heights, and hence preperiodic points, to prove Theorem~\ref{dynamicstheorem}. 

We resolve this problem using a theory of \emph{vector-valued} heights and intersections defined by Yuan-Zhang in~\cite{yz2}. These intersections take values in the completed limit of metrized, $\Q$-linearized Picard groups of models of $K$, which forms an infinite-dimensional $\Q$-vector space. When $K$ is a global field with a unique model, composing with the degree map produces the usual $\R$-valued heights. Yuan-Zhang establish this theory for fields which are finitely generated over $\Q$. Here we expand upon their work to produce an intersection theory defined relative to any finitely generated extension $K$ of an arbitrary base field $k$. 

The generalization beyond $k=\Q$ has two main benefits. First, it means that $K$ can have positive characteristic. Second, it allows exact control over what our arithmetic does and doesn't notice. We show, using model-theory results of Chatzidakis-Hrushovski~\cite{chatzidakis1,chatzidakis2} that Northcott in this relative setting is intimately related to isotriviality. The following summarizes several results from Section~\ref{heightssection}.
\begin{introtheorem}\label{intronorthcott}
Let $K/k$ be a finitely generated extension of fields, and let $f:X\to X$ be an endomorphism defined over $K$, with a polarization $f^*L=L^{\otimes q}$.
\begin{enumerate}
\item
There exists a vector-valued canonical height $h_f$, such that for $x\in X(\overline K)$, we have $h_f(f(x))=qh_f(x)$.
\item
If $k=\Q$ or $\F_q$, or if $f:X\to X$ is \emph{totally non-isotrivial} over $k$, the subset of $X(\overline K)$ with bounded degree and trivial canonical height is finite. 
\end{enumerate}
\end{introtheorem}
\emph{Totally non-isotrivial} is a necessary condition which is slightly stronger than not \emph{isotrivial}, and equivalent only for curves. This extends results of Baker~\cite{bakerisotrivial} and Benedetto~\cite{benedetto} for curves, and Lang-N\'eron~\cite{langneron} for abelian varieties.

Consider a polarizable dynamical system $f:X\to X$ defined over $\C(s,t)$. By the Lefschetz principle, we may assume this system is defined over some potentially-very-large-but-finite transcendence degree extension of $\Q$; indeed this is how Theorem~\ref{dynamicstheorem} applies to all fields $\K$. Theorem~\ref{intronorthcott} establishes, however, that heights relative to $\C(s,t)/\C$ are sufficient, provided the dynamical system is totally non-isotrivial over $\C$, and conversely, that the failure of Northcott finiteness for this height indicates some amount of isotriviality. In practice, these heights may be much more natural and practical to work with.

We also show how to compare different relative arithmetic settings. Suppose $K/k_1/k$ is a tower of finitely generated extensions. Heights relative to $K/k_1$ can be obtained as specializations of those relative to $K/k$ by intersecting with vertical fibers of a model for $K$ over a model for $k_1$, both over $\Spec k$. 

These results have immediate applications to the field of \emph{unlikely intersections} in arithmetic dynamics. Problems in this area typically study intersections of families of subvarieties, often defined to answer a dynamical question, and posit that when the intersection behaves differently than the generic case, there must be an arithmetic explanation. See \cite{ZannierUnlikely} for an overview of problems in this area. Results are often restricted to number fields or to one dimensional families, since they rely on the use of height functions with Northcott finiteness.

In some cases, one can extend results to a larger transcendence degree extension $K/k$, i.e. a higher dimensional family, by either specializing down from $K$ via transcendence degree one subfields, or by building a tower of transcendence degree one extensions up from $k$. But both require new arguments to connect each subsequent extension and to handle isotriviality, and do not always work. 

Instead, using the heights and accompanying Northcott properties of this paper, one can now generalize many of these results to higher dimensional families with little additional alteration to the method. Theorem~\ref{intronorthcott} also answers the isotriviality questions that often come up in these contexts. As an example, we can obtain the Bogomolov-type result of ~\cite[Theorem 5.2]{Ghioca2020} for unicritical polynomials in arbitrary dimension families as opposed to just for curves. This is also implied by~\cite[Theorem 1.4]{Ghioca2020b} but the proof requires a different, and more difficult argument. We expect there are numerous applications and simplifications like this. More broadly, the results presented here, via the Lefschetz principle, enable arithmetic methods to be used over all fields, instead of just number fields or function fields of curves.

\textbf{Acknowledgements}
The author thanks Dragos Ghioca, Tom Tucker, and Xinyi Yuan for several helpful discussions while preparing this paper. Many of the ideas presented here come from Yuan and Zhang's work~\cite{yz2}, and the author thanks them both for their insight, explanations, and encouragement to complete this work.

\subsection{Vector-valued heights and intersections}
Throughout this paper, let $K$ be a finitely generated extension of an arbitrary field $k$. We divide this into two settings, generalizing the distinction between number fields and function fields. 
\begin{itemize}
\item\textbf{The base $\Z$ case}: Here $k=\Q$ and we define $d\ge0$ to be the transcendence degree of $K$ over $k=\Q$. As the name implies, the base scheme for all models in this case will be $\Spec \Z$.
\item\textbf{The base $k$ case}: In this case $k$ is allowed to be any field, and we define $d\ge0$ such that $K$ has transcendence degree $d+1$ over $k$. The base scheme for all models will be $\Spec k$.
\end{itemize}
The case $d=0$ corresponds to $K$ being a number field or a transcendence degree one function field over $k$, in which case most of the results presented here are already established. 

\begin{remark}
In characteristic zero, the two cases are not always mutually exclusive. For example, the fields $K=\Q(t)$ and $k=\Q$ could be studied with $d=1$ (base $\Z$), or with $d=0$ (base $k$). These two perspectives on the same field will produce different (though related; see Section~\ref{heightscompare}) heights and intersection numbers. Thus if $k=\Q$, to clarify the intersection theory under consideration, one must specify whether one's perspective is the base $\Z$ or base $k$ case, or equivalently, the value of $d$. 
\end{remark}

\begin{definition}
Let $X$ be a geometrically normal projective variety over $K$.
\begin{enumerate}
\item
A \emph{projective arithmetic variety} (resp. \emph{open arithmetic variety}) is an integral scheme which is projective (resp. quasi-projective) over $\Spec k$ in the base $k$ case, or projective (resp. quasi-projective) and flat over $\Spec \Z$ in the base $\Z$ case. 

\item
Given an open arithmetic variety $\cU$, a \emph{projective model} for $\cU$ is an open embedding $\cU\hookrightarrow\cX$ into a projective arithmetic variety $\cX$, whose complement $\cX\backslash\cU$ is the support of an effective Cartier divisor.
\item
A \emph{projective (resp. open) arithmetic model} for $K$ is a projective (resp. open) arithmetic variety whose function field is $K$.
\item
A \emph{projective (resp. open) arithmetic model} for $X/K$ is a projective (resp. open) arithmetic model $\cV$ for $K$, together with a projective and flat morphism $\cU\to\cV$ whose generic fiber is $X\to \Spec K$.
\end{enumerate}
\end{definition}

Fix an open arithmetic variety $\cU$. In Section~\ref{defssection} we construct the group $\widehat\Pic(\cU)_{\Mod}$ as the limit over projective arithmetic models $\cX$ for $\cU$ of either $\Pic(\cX)_{\Q}$ (in base $k$), or $\widehat\Pic(\cX)_{\Q}$, the group of $\Q$-line bundles with $C^{\infty}$-Hermitian metrics (in base $\Z$). This is given a topology and completed to $\widehat\Pic(\cU)_{\Cont}$.  

Now let $X$ be a geometrically normal $K$ variety, and define the group of adelic line bundles on $X$,
\[
\widehat\Pic(X)_{\Cont}:=\lim_{\cU\to\cV}\widehat\Pic(\cU)_{\Cont},
\]
where this limit is taken over open arithmetic models $\cU\to\cV$ for $X\to\Spec K$. We additionally define $\widehat\Pic(K)_{\Cont}:=\widehat\Pic(\Spec(K))_{\Cont}.$ 

The subset $\widehat\Pic(X)_{\Nef}\subset\widehat\Pic(X)_{\Cont}$ of \emph{nef} adelic line bundles consists of those which come from limits of nef (Hermitian) line bundles on projective arithmetic models, and the subgroup $\widehat\Pic(X)_{\Int}$ of \emph{integrable} adelic line bundles consists of differences of nef adelic line bundles. On this subgroup, we show that the Deligne pairing~\cite{Deligne} on $\widehat\Pic(\cU)$ converges with respect to the above limits, and defines a vector-valued relative intersection product on $X$, and $\R$-valued product on $K$:
\[
\widehat\Pic(X)_{\Int}^{n+1}\longrightarrow\widehat\Pic(K)_{\Int},
\qquad\qquad
\widehat\Pic(K)_{\Int}^{d+1}\longrightarrow\R.
\]

We also establish comparative relationships:
\begin{definition}\label{positivitydefs}
Let $\overline L,\overline M\in\widehat\Pic(X)_{\Int}$ and let $\overline H_1,\overline H_2\in\widehat\Pic(K)_{\Int}$. We define:
\begin{enumerate}
\item
$\overline H_1$ is \emph{pseudo-effective}, and we write $\overline H_1\ge0$ provided that $\overline H_1\cdot \overline N_1\cdots\overline N_d\ge0$ for any $\overline N_1,\dots,\overline N_d\in\widehat\Pic(K)_{\Nef}$. We write $\overline H_1\ge \overline H_2$ to mean that $\overline H_1-\overline H_2$ is pseudo-effective.
\item
$\overline H_1$ is \emph{numerically trivial}, written $\overline H_1\equiv0$ provided that $\overline H_1\cdot \overline N_1\cdots\overline N_d=0$ for any $\overline N_1,\dots,\overline N_d\in\widehat\Pic(K)_{\Int}$. We write $\overline H_1\equiv\overline H_2$ to mean $\overline H_1-\overline H_2\equiv0$, and then $\overline H_1\equiv \overline H_2$ if and only if $\overline H_1\ge\overline H_2$ and $\overline H_2\ge \overline H_1$.
\item

$\overline L$ is \emph{arithmetically positive}, written $\overline L\gg0$, provided that $L_1$ is ample and $\overline L-\pi^*\overline N\in\widehat\Pic(X)_{\Nef}$ for some $\overline N\in\widehat\Pic(k_1)_{\Int}$ with arithmetic degree $\widehat\deg(\overline N)>0$, where $k_1$ is any subfield of $K$ such that $K$ has transcendence degree $d$ over $k_1$. Since this makes $k_1$ either a number field or a transcendence degree one function field over $k$, the arithmetic degree is defined as in~\cite{carney,yz}. We use the same convention as in (1) to define $\overline L\gg\overline M$.
\item $\overline L$ is \emph{vertical}, notated $\oL\in\widehat\Pic(X)_{\Vrt}$, provided that $L=\cO_X$.

\item
$\overline M$ is $\overline L$-\emph{bounded} if there exists $\epsilon>0$ such that both $\overline L\pm\epsilon\overline M\in\widehat\Pic(X)_{\Nef}$.

\end{enumerate}
\end{definition}

For $\oL\in\widehat\Pic(X)_{\Int}$ and $Z$ an integral closed $\overline K$-subvariety of $X$, write $\widetilde Z$ for the minimal $K$-subvariety of $X$ which contains $Z$, and define the height of $Z$ as
\[
h_{\oL}(Z):=\frac{\left(\oL|_{\widetilde Z}\right)^{\dim Z+1}}{\left(\dim Z+1\right)\left(L|_{\widetilde Z}\right)^{\dim Z}}\in\widehat\Pic(K)_{\Int}. 
\]
When $Z\in X(\overline K)$, we say $(L|_{\widetilde Z})^{\dim Z}:=\deg(Z)$. Given a polarized dynamical system $(X,f,L)$, we will construct a unique $f^*$-equivariant lift $\oL_f\in\widehat\Pic(X)_{\Int}$, which produces the canonical height $h_f=h_{\oL_f}$. 

We define the following, after Chatzidakis-Hrushovski~\cite{chatzidakis1,chatzidakis2}. 
\begin{definition}\label{totallynonisotrivial}
A \emph{constructible map} consists of a composition of rational maps and inverses of purely inseparable rational maps. A dynamical system $f:X\to X$ defined over $K$ is \emph{constructibly isotrivial} over $k$ provided that there exists an irreducible projective variety $Y$ with a dominant rational map $g:Y\to Y$ defined over $k$, and an isomorphism of constructible maps $(X,f)\cong(Y_K,g_K)$, i.e. 
\[
\begin{tikzcd}
X\arrow[shift left=.5ex]{r}{\phi}
&
Y_K \arrow[shift left=.5ex]{l}{\psi}
\end{tikzcd}
\]
such that $\phi$ and $\psi$ are constructible, inverse to each other, and $\psi\circ g\circ\phi=f$. We call $f:X\to X$ \emph{totally non-isotrivial} provided that it is not itself constructibly isotrivial and further no positive dimensional periodic closed subvariety of $X$ is constructibly isotrivial.
\end{definition}

Theorem~\ref{intronorthcott} states 
that if $k$ is finite, if $k=\Q$, or if $f:X\to X$ is \emph{totally non-isotrivial} over $k$, the set
\[
\left\{x\in X(\overline K):h_{f}(x)\equiv0,\text{ and }[K(x):K]\le D\right\}
\]
is finite.

\subsection{The Hodge Index Theorem}
Classically, the Hodge Index Theorem states that the intersection pairing on a projective algebraic surface has signature $+1,-1,\dots,-1$. Faltings~\cite{faltings} and Hriljac~\cite{hriljac} independently proved the same for arithmetic surfaces, by equating the arithmetic intersection product to the negative of the N\'eron-Tate height pairing on the Jacobian of the generic fiber, which was then generalized by Moriwaki~\cite{Moriwaki} to higher dimensional arithmetic varieties. Here we extend the work of Yuan-Zhang~\cite{yz} and the author~\cite{carney} for adelic line bundles to arbitrary finitely generated fields.

In base $k$, we require Chow's trace functor. Suppose $K/k$ falls into the base $k$ setting and $A$ is an abelian $K$-variety. 
The \emph{$K/k$-trace} $\Tr_{K/k}(A)$ is an abelian variety defined over $k$ along with a trace morphism 
\[
\tau:\Tr_{K/k}(A)\longrightarrow A
\]
with the universal property that any morphism $B_K\to A$, where $B$ is an abelian variety defined over $k$, must factor through $\tau$. In characteristic zero, $\tau$ is an injection and embeds $\Tr(A)_{K/k}$ as a subgroup of $A$, but in positive characteristic it may have a finite infinitesimal kernel (this is one explanation for the use of constructible maps above). Let $\Pic^0_{\Red,X}$ be the reduction of the neutral component of the Picard variety of $X$. We will show that every 
\[
L\in\Tr_{K/k}\Pic^0(X):=\Tr_{K/k}\left(\Pic^0_{\Red,X}\right)(k)
\]
 can be given unique additional structure as an adelic line bundle inside $\widehat\Pic(X)_{\Int}$. In base $\Z$, since there is no field underlying $\Z$, we set $\Tr_{K/k}=0$ by default. We prove the following theorem. 

\begin{introtheorem}\label{hodgeindex}
Let $X/K$ be a geometrically normal, geometrically connected projective variety of dimension $n\ge1$. Let $\overline M\in\widehat\Pic(X)_{\Int}$ and $\oL_2,\dots,\oL_n\in\widehat\Pic(X)_{\Nef}$ such that each $L_i$ is big and $M\cdot L_2\cdots L_n=0$ on $X$ if $n>1$, or $\deg M=0$ when $X$ is a curve. Then
\begin{enumerate}
\item \emph{(Inequality)}
\[
\overline M^2\cdot\oL_2\cdots\oL_n
\le0.
\]
\item \emph{(Equality)}
Further, if $\oL_i\gg0$ and $\overline M$ is $\oL_i$-bounded for all $i$,
\[
\overline M^2\cdot\oL_2\cdots\oL_n\equiv0
\]
if and only if 
\[
\overline M\in\pi^*\widehat\Pic(K)_{\Int}+\Tr_{K/k}\Pic^0(X).
\]
\end{enumerate}
\end{introtheorem}
This is proven by Yuan-Zhang~\cite{yz2} when $K$ is a finitely generated extension of $\Q$ (note that the trace doesn't appear in that case), and by the author~\cite{carney} when $K$ is a transcendence degree one extension of an arbitrary field $k$.

\subsection{Outline of paper}
Section~\ref{defssection} makes precise the definitions of the groups $\widehat\Pic(X)_{\Cont}$ and $\widehat\Pic(K)_{\Cont}$, their topologies, and their named subsets and subgroups. Additionally, we present a generalization of Berkovich's analytic spaces, and use this to connect adelic line bundles and metrized line bundles. Section~\ref{intersectionssec} then recalls the Deligne pairing, and constructs a relative intersection product for adelic line bundles. We also present a projection formula and show how to relate this relative product to the classical (in base $k$) and arithmetic (in base $\Z$) intersection pairings which take values in $\R$.

Next, given a polarizable endomorphism $f:X\to X$, Section~\ref{admissible} constructs particular lifts of line bundles $L\in\Pic(X)$ to adelic line bundles $\oL_f\in\widehat\Pic(X)_{\Int}$ in an $f^*$-equivariant way. These are the adelic line bundles used to define canonical heights. Section~\ref{trace} details the $K/k$-trace construction, which produces an important example of such a lift.

Section~\ref{heightssection} proves Theorem~\ref{intronorthcott}, as well as several related corollaries about vector-valued heights and canonical heights. Section~\ref{essentialminimumsec} generalizes Zhang's essential minimum~\cite[Theorem 1.11]{Z95} to the context of this paper, providing a relationship between heights of subvarieties and heights of points on those subvarieties.

Section~\ref{hodgeindexsection} proves the Hodge Index Theorem, Theorem~\ref{hodgeindex}, beginning with the inequality, then proving the equality for vertical $\oM$, for when $X$ is a curve, and finally in general. Section~\ref{dynamicssection} extends the results of Section~\ref{admissible} to construct \emph{admissible} metrics, then applies the Hodge Index Theorem to prove Theorem~\ref{dynamicstheorem} and Corollary~\ref{introequilibrium}.

\section{Adelic line bundles}\label{defssection}




This section follows~\cite[Section 2]{yz2} in the base $\Z$ case, with some additional considerations in the base $k$ case. As a notational convenience so that we can work in the base $k$ and base $\Z$ settings simultaneously, we define the following.
\begin{definition}
Let $\cX$ be a projective arithmetic variety. In the base $k$ setting, we simply define $\widehat\Div(\cX):=\Div(\cX)$ and $\widehat\Pic(\cX):=\Pic(\cX)$ to be the usual divisor and Picard groups. In the base $\Z$ setting, $\widehat\Div(\cX)$ is defined to be the arithmetic divisor group, whose elements are of the form $\overline\cD=(\cD,g)$, where $\cD$ is a divisor on $\cX$ and $g$ is a Green's function on $\cX(\C)$. Define $\widehat\Pic(\cX)$ to be the group of $C^{\infty}$-Hermitian line bundles on $\cX$. 
\end{definition}
We will write elements $\overline\cD\in\widehat\div(\cX)$ and $\overline\calL\in\widehat\Pic(\cX)$ in both settings, where all properties of the Green's function or the Hermitian metric are vacuously true in the base $k$ setting.

$\widehat\Div(\cX)$ is partially ordered by effectivity: an arithmetic divisor $(\cD,g)$ is \emph{effective} (resp. \emph{strictly effective}) if $\cD$ is effective (resp. strictly effective, i.e. effective and non-zero) and $g\ge0$ (resp. $g>0$) away from the support of $\cD(\C)$ on $\cX(\C)$. Then for $\overline\cD,\overline\cE\in\widehat\Div(\cX)$, writing $\overline\cD<\overline\cE$ when $\overline\cE-\overline\cD$ is strictly effective induces a partial ordering.

If $d=0$, there is a unique choice of projective model for $K$ in the base $k$ case, namely the unique projective curve whose function field is $K$. For larger $d$, however, there is no one unique choice, so in order to work in full generality, we take limits over all possible models.

\begin{definition}
Let $\cU$ be an open arithmetic variety. Define
\[
\widehat\Div(\cU)_{\Mod}:=\lim_{\overrightarrow{\cX}}\widehat{\Div}(\cX)_{\Q},\qquad \widehat\Pic(\cU)_{\Mod}:=\lim_{\overrightarrow{\cX}}\widehat{\Pic}(\cX)_{\Q}
\]
via pullbacks on the inverse system of projective models $\cU\hookrightarrow\cX$.
\end{definition}

If $\cX$ and $\cX'$ are two different projetive models, there exists a third projective model which dominates both, and we pull back divisors on $\cX$ and $\cX'$ to this third model to compare them. Thus the partial ordering above defines a topology on $\widehat\Div(\cU)_{\Mod}$. Fix a projective model $\cU\hookrightarrow\cX$ and an effective Cartier divisor $\cD$ with support equal to $\cX\backslash\cU$, and let $g$ be any Green's function of $\cD$ such that $(\cD,g)$ is strictly effective. Then the sets
\[B(\epsilon,0):=\{\overline\cE\in\widehat\Div(\cU)_{\Mod}|-\epsilon\overline\cD<\overline\cE<\epsilon\overline\cD\},\quad \epsilon\in\Q_{>0},
\]
form a basis of open sets at 0, and thus by translation, at every element of $\widehat\Div(\cU)_{\Mod}$. Further, this topology doesn't depend on the choice of projective model or divisor. If $\cX'$ is a different model and $\overline\cD'$ a different divisor with the above properties, we have 
\[
\frac1\alpha\overline\cD<\overline\cD'<\alpha\overline\cD
\]
for some $\alpha>1$.

This allows us to make the following definition:
\begin{definition}
$\widehat\Div(\cU)_{\Cont}$ is defined to be the completion of $\widehat\Div(\cU)_{\Mod}$ with respect to the above topology. We further define $\widehat\Prin(\cU)_{\Cont}$ to be the completion of the subgroup of principal arithmetic divisors in $\widehat\Div(\cU)_{\Mod}$, and then define
\[
\widehat\Pic(\cU)_{\Cont}:=\widehat\Div(\cU)_{\Cont}/\widehat\Prin(\cU)_{\Cont}.
\]
This is called the group of \emph{adelic line bundles} on $\cU$.
\end{definition}

Equivalently, we can construct $\widehat\Pic(\cU)_{\Cont}$ as the completion of $\widehat\Pic(\cU)_{\Mod}$ as follows. Let 
$
\left(\overline\calL_i\right)_{i\ge1}
$
be a sequence of (Hermitian, in the base $\Z$ setting) line bundles on projective models $\cX_i$ for $\cU$, with 
\begin{enumerate}
\item
compatible dominating morphisms $\pi_{i,j}:\cX_j\to\cX_i$ for each pair $j\ge i\ge1$, and 
\item
a compatible system $\{\ell_{ij}\}$ of rational sections of $\calL_j\otimes\pi^*_{i,j}\calL_i^{-1}$ such that $\div(\ell_{i,j})$ is supported on $\cX_j\backslash\cU$ for all $j\ge i\ge 1$. 
\end{enumerate}

Fix a strictly effective divisor $\overline\cD$ with support $\cX\backslash\cU$. Such a system converges provided that for any $\epsilon>0$, there exists an $i_0$ such that when $j\ge i\ge i_0$, 
\[
-\epsilon\pi^*_{1,j}\overline\cD<\div(\ell_{i,j})<\epsilon\pi^*_{1,j}\overline\cD.
\]
A sequence converges to zero if there is a rational section $\ell_i$ of each $\overline\calL$ such that $\ell_{i,j}=\ell_j\otimes\pi^*_{i,j}\ell_i^{-1}$ and such that when $i\ge i_0$, 
\[
-\epsilon\pi^*_{1,i}\overline\cD<\div(\ell_{i})<\epsilon\pi^*_{1,i}\overline\cD.
\]
$\widehat\Pic(\cU)_{\Cont}$ is the set of such convergent sequences modulo those which converge to zero. Like the topology defined on $\widehat\Div(\cU)_{\Mod}$, this doesn't depend on the choice of $\overline\cD$. The partial ordering by effectivity extends to this completion.

We prove a lemma justifying calling this the completion of $\widehat\Pic(\cU)_{\Mod}$.

\begin{lemma}\label{completion}
The map $\widehat\Pic(\cU)_{\Mod}\to\widehat\Pic(\cU)_{\Cont}$ is injective.
\end{lemma}

\begin{proof}
Fix a projective model $\cX$ and a strictly effective arithmetic divisor $\overline\cD$ whose support is $\cX\backslash\cU$. 
Suppose $\overline\cE\in\widehat\Pic(\cU)_{\Mod}$ is such that for all $\epsilon>0$,
$
\epsilon\overline\cD\pm\overline\cE
$
are both strictly effective. Without loss of generality we may assume $\overline\cE$ is also realized on $\cX$. Then, letting $n=\dim\cU$, and letting $\overline\calL_1,\dots,\calL_{n-1}$ be ample (Hermitian) line bundles in $\Pic(\cX)$, we have
\[
\left(\epsilon\overline\cD\pm\overline\cE\right)\cdot\overline\calL_1\dots\overline\calL_{n-1}>0
\]
for all $\epsilon>0$, which implies
\[ 
\overline\cE\cdot\overline\calL_1\dots\overline\calL_{n-1}=0.
\]
By Moriwaki's Hodge Index Theorem for arithmetic divisors~\cite{Moriwaki} in the base $\Z$ case, and the classical Hodge Index Theorem in the base $k$ case~\cite[Exp. XIII, Cor. 7.4]{SGA6}, we must have $\overline\cE=0$.
\end{proof}

\begin{definition}
An adelic line bundle $\overline\calL\in\widehat\Pic(\cU)_{\Cont}$ is called \emph{nef} provided that it can be represented by a convergent sequence $\{\overline\calL_i\}_{i\ge1}$ where each $\calL_i$ is nef (for the definition of a nef Hermitian line bundle, see~\cite{Z95}). We call
$\overline\calL$ \emph{integrable} provided it can be written as the difference of two nef adelic line bundles. The subset of nef elements, and subgroup of integrable elements of $\widehat\Pic(\cU)_{\Cont}$ are denoted $\widehat\Pic(\cU)_{\Nef}$ and $\widehat\Pic(\cU)_{\Int}$, respectively.
\end{definition}

Having defined the divisor and Picard groups of limits of projective models of open arithmetic varieties, we now take limits over the inverse system of open arithmetic models for the projective variety $X/K$.

\begin{definition}
We define
\[
\widehat\Pic(X)_{\Cont}:=\lim_{\overrightarrow{\cU\to\cV}}\widehat\Pic(\cU)_{\Cont},
\]
where the limit is taken over open arithmetic models $\cU\to\cV$ for $X\to\Spec K$. When $X=\Spec K$, we write simply $\widehat\Pic(K)_{\Cont}$. We also define the groups $\widehat\Pic(X)_{\Mod}$ and $\widehat\Pic(X)_{\Int}$, and the set $\widehat\Pic(X)_{\Nef}$, via the same corresponding limits.
\end{definition}

We will write elements as $\overline L\in\widehat\Pic(X)_{\Cont}$. While, strictly speaking, $\overline L$ is a projective limit of convergent sequences of line bundles $\ocL_i$ on models, it makes sense to consider the \emph{generic fiber} $L$ on $X$, as $\overline\calL_{i,K}$ are all equal.

\begin{definition}
$\overline L\in\widehat\Pic(X)_{\Int}$ is called \emph{vertical} provided that $L\cong\cO_X$. Denote by $\widehat\Pic(X)_{\Vrt}$ the subgroup of vertical adelic line bundles.
\end{definition}

\subsection{Models over intermediate fields}\label{basek1}
Suppose we have an intermediate field $K/k_1/k$. Fix open arithmetic models $\cV$ for $K$ and $\cW$ for $k_1$. Since we have a natural morphism $\Spec K\to\Spec k_1$, after possibly replacing $\cV$ by an open sub-model we have a morphism $\pi:\cV\to\cW$. Now fix a projective model $\cC$ for $\cW$. By Nagata's compactification theorem~\cite{Nagata} (see~\cite{conradnagata} for a modern treatment), $\pi$ extends to a projective morphism $\pi:\cB\to\cC$ of projective arithmetic models for $K$ and $k_1$. This is flat over a Zariski open locus, and then by Raynaud-Gruson's flattening theorem~\cite{raynaudflattening}, after replacing $\cC$ with a blowup and $\cB$ with its strict transform, we may assume $\cB\to\cC$ is flat.

Given any two projective arithmetic models (for a field or a projective variety), we can find a third projective arithmetic model which dominates both. Thus, we may assume without loss of generality that every projective arithmetic model $\cX$ for $X$ is fibered over some projective arithmetic model $\cB$ for $K$, which is in turn fibered over some projective arithmetic model $\cC$ for $k_1$. In particular, it makes sense to consider the base change $\cX_{k_1}$, and we can study this from two different viewpoints.

First, consider the finitely generated extension $k_1/k$. Then $\cX_{k_1}$ is a projective variety over $k_1$, and $\cX\to\cC$ is a projective arithmetic model. Second, consider the finitely generated extension $K/k_1$. Then $X$, as before, is a projective variety over $K$, but now $\cB_{k_1}$ is a projective arithmetic model for $K$ over $k_1$, and $\cX_{k_1}\to\cB_{k_1}$ is a projective arithmetic model for $X\to\Spec(K)$. Note that if originally $K/k$ falls into the base $\Z$ case, the first viewpoint will still be base $\Z$, but the second is now base $k$ (meaning that the base is the spectrum of a field; the specific base field is $k_1$ not $k$).

The first viewpoint provides a proof technique which we will use often, effectively trading the dimension of our projective variety against the transcendence degree of our field extension. The second viewpoint provides a way to adjust the granularity of the intersection theory to be defined in the next section. This will be discussed further in Section~\ref{heightscompare}.




\subsection{Metrized line bundles and analytic spaces}\label{analytic}

We now develop the analytic theory, so that adelic line bundles can also be treated as metrized line bundles. This is a slight generalization of Berkovich's analytic spaces,~\cite{Berk}.

Fix an intermediate field $k_1$ such that $K$ has transcendence degree $d$ over $k_1$, in order to specify a one dimensional set of places. In the base $\Z$ case $k_1$ must be $\Q$. In the base $k$ case, $k_1$ can be any subfield of $K$ which has transcendence degree one over $k$. In base $k$ the Analytic spaces defined in this subsection depend heavily on this choice, but the intersection products and heights which follow, which can be written using analytic language, do not.

\begin{definition}
Define $(\Spec k_1)^{\an}$ to be the topological space consisting of the places of $k_1$, with the discrete topology. Associate to each $\nu\in(\Spec k_1)^{\an}$ a normalized non-trivial absolute value $|\cdot|_{\nu}$ such that the product formula
\[
\prod_{\nu\in(\Spec k_1)^{\an}}|\alpha|_{\nu}=1
\]
holds for all $\alpha\in k_1$.
\end{definition}

$(\Spec\Q)^{\an}$ corresponds to the set $\{\infty,2,3,5,7,\dots\}$ with the usual normalized Archimedean and $p$-adic absolute values. In the base $k$ case, $k_1$ is the function field of a unique projective curve over $k$, and $(\Spec k_1)^{\an}$ corresponds to the closed points of this curve, with the discrete topology instead of the Zariski topology.

Now let $X$ be a $k$ scheme, not necessarily of finite type. The \emph{Berkovich analytification} of $X$ relative to $k_1$ is a disjoint union 
\[
(X/k_1)^{\an}:=\coprod_{\nu\in\Spec(k_1)^{\an}}X_{\nu}^{\an},
\]
defined via the following properties. When $k_1$ has already been established, in particular when $k_1=\Q$ in the base $\Z$ setting, we will often write $X^{\an}$ for $(X/k_1)^{\an}$. 
\begin{itemize}
\item
Let $\Spec A$ be an affine scheme. Then $(\Spec A)_{\nu}^{\an}$, as a set, consists of all multiplicative semi-norms $|\cdot|_x$ on $A$ extending $|\cdot|_{\nu}$ on $k_1$. Given $x\in(\Spec A)^{\an}_{\nu}$ and $f\in A$, we write
\[
|f(x)|_{\nu}:=|f|_x.
\]
If $X$ is covered by open affine schemes $\Spec A$, then $X_{\nu}^{\an}$ is covered by open affinoid sets $(\Spec A)_{\nu}^{\an}$.
\item
$(\Spec A)^{\an}_{\nu}$ is given the weakest topology such that every function ${|\cdot|_x:A\to\R}$ is continuous, and $X^{\an}$ is given the weakest topology such that every $(\Spec A)^{\an}_{\nu}$ is open.
\item
There is a canonical map $X^{\an}\to X$, which takes a semi-norm $|\cdot|_x$ to the prime ideal defined by its kernel.
\item
A map of $k$-schemes $f:X\to Y$ has an analytification $f^{\an}:X^{\an}\to Y^{\an}$.
\end{itemize}

The properties of $X^{\an}_{\nu}$ are established in~\cite{Berk}, and~\cite{yz2} extends this to the disjoint union over all $\nu$.

\begin{remark}
$\Spec(A)^{\an}$ is not the same as the Berkovich spectrum of $A$, typically denoted $\mathcal M(A)$. For example $\mathcal M(\Z)$ contains distinct points for every choice of normalization of each absolute value. If $X$ is of finite type over $k_1$ and $\nu$ is non-archimedean, however, $X^{\an}_{\nu}$ is the usual Berkovich analytification of the base change of $X_{k_{1,\nu}}$.
\end{remark}

As an important example, suppose $X$ is an integral scheme with function field $F$. Then $\Spec F^{\an}$ consists of the fiber over the generic point of $X$ under $X^{\an}\to X$, which we call the \emph{generic points} of $X^{\an}$. As is true for schemes, we have:

\begin{lemma}\label{berkgenericpt}
$\Spec F^{\an}$ is dense in $X^{\an}$.
\end{lemma}
\begin{proof}
Now extend $X\to\Spec K$ to a finite type morphism $\cX\to\cV$ of $k_1$ varieties such that the function field of $\cV$ is $K$. Taking the analytification, we have injections
\[
(\Spec F)^{\an}\hooklongrightarrow X^{\an}\hooklongrightarrow\cX^{\an}
\]
which consist of disjoint unions of injections
\[
(\Spec F)^{\an}_{\nu}\hooklongrightarrow X^{\an}_{\nu}\hooklongrightarrow \cX^{\an}_{\nu}
\]
at every place $\nu\in(\Spec k_1)^{\an}$. But now $\cX_{\nu}^{\an}=\cX_{(k_1)_{\nu}}^{\an}$ is the usual Berkovich analytification of $\cX_{(k_1)_{\nu}}$ over $(k_1)_{\nu}$, and so $\Spec F_{\nu}^{\an}$ is dense in $\cX_{\nu}^{\an}$, which completes the proof.
\end{proof}

We can extend our definition to projective arithmetic models $\cX$ for $X$ by specifying that
\[
(\cX/k_1)^{\an}:=\left(\cX_{k_1}/k_1\right)^{\an}.
\]
by the previous lemma, $X^{\an}$ is dense in $\cX^{\an}$.

Now suppose $L$ is a line bundle on $X$. We write $L^{\an}$, the \emph{analytification} of $L$, to denote the line bundle on $X^{\an}$ whose fiber $L^{\an}(x)$ for every $x\in X^{\an}$ is the same as that of $L$ over the image of $x$ under $X^{\an}\to X$. A \emph{continuous metric} on $L^{\an}$ consists of a norm $||\cdot||_x$ on $L^{\an}(x)$ for every $x\in X^{\an}$ such that 
\begin{itemize}
\item $||f\ell||_x=|f|_x\cdot||\ell||_x$ for every $\ell\in L^{\an}(x)$ and $f$ in the residue field of $x\in X^{\an}$.
\item
For every rational section $s$ of $L$ defined on an open set $U\subset X$, the function $x\mapsto||s(x)||_x$ is continuous on $U^{\an}$.
\end{itemize}
We denote a \emph{continuously metrized line bundle} as $\overline L:=(L,||\cdot||)$, where $||\cdot||$ is a continuous metric on $L^{\an}$. Write $\widehat\Pic(X^{\an})_{\Cont}$ for the group of isometry classes of continuously metrized line bundles on $X$.

Given a section $s$ of a continuously metrized line bundle $\overline L$ on $X$ with divisor $D=\div(s)$, the function $g(x):=-\log||s(x)||_x$ defines a \emph{Green's function} for $D^{an}$, meaning that $g$ is continuous outside of $D^{\an}$ and has a logarithmic singularity on $D^{\an}$ on $X^{\an}$. If instead we start with a divisor $D$ on $X$ and a Green's function $g$ on $D^{\an}$, the function $e^{-g}$ defines a metric on $\cO(D)$. Such a pair $\overline D:=(D,g)$ is called an arithmetic divisor. The group of such is denoted $\widehat\Div(X^{\an})_{\Cont}$, and the above construction induces a homomorphism
\[
\widehat\Div(X^{\an})_{\Cont}\longrightarrow\widehat\Pic(X^{\an})_{\Cont}.
\]

Let $\cX\to\cB$ be a projective arithmetic model for $X/K$, and  let $\overline\calL$ be a (Hermitian) line bundle on $\cX$ with generic fiber $L=\overline\calL_K$. As described in~\cite{carney,yz}, $\overline\calL$ induces a metric of $\overline\calL^{\an}$, which in turn induces a \emph{model metric} on $L^{\an}$. Thus we have a map
\[
\widehat\Pic(\cX)\to\widehat\Pic(\cX^{\an})_{\Cont}\to\widehat\Pic(X^{\an})_{\Cont}.
\]
Since $X^{\an}$ is dense in $\cX^{\an}$, this is an injection.

We can now tie metrized line bundles and adelic line bundles together:

\begin{lemma}
There exist natural injections
\[
\widehat\Pic(X)_{\Mod}\hooklongrightarrow\widehat\Pic(X)_{\Int}\hooklongrightarrow\widehat\Pic(X)_{\Cont}\hooklongrightarrow\widehat\Pic(X^{\an})_{\Cont}
\]
\end{lemma}
\begin{proof}
It suffices to show that the equivalent morphisms exist for an open arithmetic model $\cU$ of $X$. The first two injections are already established by Lemma~\ref{completion}. We also have an injection $\widehat\Pic(\cU)_{\Mod}\hookrightarrow\widehat\Pic(\cU^{\an})_{\Cont}$ given by the induced model metrics. Fix a strictly effective divisor $\overline\cD$ giving the topology on $\widehat\Pic(\cU)_{\Mod}$. Then the image of $B(\epsilon,0)\subset\widehat\Pic(\cU)_{\Mod}$ in $\widehat\Pic(\cU^{\an})_{\Cont}$ is the space of continuous metrics $||\cdot||$ on $\cO_{\cU}$ such that $-\log||1||$ is bounded by $\epsilon g_{\overline\cD}$, where $g_{\overline\cD}$ is the Green's function corresponding to $\overline\cD$. Then $\widehat\Pic(\cU^{\an})_{\Cont}$ is complete with respect to the topology induced by this image. Thus, the map
\[
\widehat\Pic(\cU)_{\Mod}\hooklongrightarrow\widehat\Pic(\cU^{\an})_{\Cont}
\]
extends to the completion. 
\end{proof}

From now on, we will interchangeably use the terms \emph{integrable adelic line bundles} and \emph{integrable metrized line bundles} to refer to both $\widehat\Pic(X)_{\Int}$ and its image in $\widehat\Pic(X^{\an})$.




\section{Intersections of adelic line bundles}\label{intersectionssec}

\subsection{Vector-valued intersections}
For the following subsection, let $X\to Y$ be a flat morphism of projective $K$ varieties of pure relative dimension $n$. Often we will consider $Y=\Spec K$, but the following holds in general. We can find open arithmetic models with compatible projective flat morphisms 
\[
\cU\longrightarrow\cW\longrightarrow\cV
\]
for $X\to Y\to\Spec K$.

Fix a projective model $\cY$ for $\cW$. Applying Nagata's compactification Theorem~\cite{Nagata} (see~\cite{conradnagata} for a modern treatment), we can find projective model $\cX$ for $\cU$ with a map $\cX\to\cY$ extending $\cU\to\cW$. By Raynaud-Gruson's flattening theorem~\cite{raynaudflattening} we may further assume $\cX\to\cY$ is flat, after replacing $\cY$ by a blow up (outside of $\cW$) and $\cX$ by its strict transform. Thus we have a projective arithmetic model $\cX\to \cY$ extending the open arithmetic model $\cU\to\cW$.

Deligne defines a multilinear product 
\[
\widehat\Pic(\cX)^{n+1}\longrightarrow\widehat\Pic(\cY)
\]
\[
\left(\overline\calL_1,\cdots,\overline\calL_{n+1}\right)\mapsto\langle\overline\calL_1,\cdots,\overline\calL_{n+1}\rangle(\cX/\cY).
\]
The construction and properties of this product are given in~\cite{Deligne}, with the base $k$ case (i.e a product for line bundles) in section 8.1, and base $\Z$ case (for line bundles with Hermitian metrics) in section 8.3. Additional details can be found in~\cite{zhangdeligne}. 

\begin{lemma}
Deligne's pairing induces a multilinear product
\[
\langle\cdot,\dots,\cdot\rangle(X/Y):\widehat\Pic(X)_{\Int}^{n+1}\longrightarrow\widehat\Pic(Y)_{\Int}.
\]
\end{lemma}

\begin{remark}
While Deligne's product is defined on every $\widehat\Pic(\cX)$, here we only use the integral subset of $\widehat\Pic(X)_{\Cont}$, as the the limit of the pairing over a convergent sequence of line bundles may not converge in general; see~\cite{Z95} and~\cite{moriwakifg2} in the $d=0$ case.
\end{remark}

\begin{proof}
We follow~\cite{yz2}, extending their proof to the more general relative product used here. 

It suffices to show that there is an induced product
\[
\widehat\Pic(\cU)_{\Int}^{n+1}\longrightarrow\widehat\Pic(\cW)_{\Int},
\]
as both $\widehat\Pic(X)_{\Int}$ and $\widehat\Pic(Y)_{\Int}$ are defined as the projective limits of these. By linearity, it suffices to prove this when $\overline\calL_0,\dots,\overline\calL_{n}\in\widehat\Pic(\cU)_{\Int}$ are nef.

For $0\le i\le n$, let $\left\{\overline\calL_{i,m}\right\}_{m\ge0}$ be a sequence of nef (Hermitian) line bundles on projective arithmetic models $\pi_m:\cX_m\to \cY_m$, which converges to $\overline\calL_i$ in $\widehat\Pic(\cU)_{\Int}$. Note that by moving to a dominating model we may assume, as the notation suggests, that $\cX_m\to \cY_m$ does not depend on $i$, and that each $\pi_{m'}$ dominates $\pi_m$ for $m'\ge m$.

Fix a strictly effective divisor $\overline\cD$ on $\cY_0$ with support outside $\cW$, so that $\overline\cD$ defines the topology on $\widehat\Pic(\cW)_{\Cont}$, and $\pi^*\overline\cD  $ defines the topology on $\widehat\Pic(\cU)_{\Cont}$. For every $m'\ge m\ge0$, let $\ell_{i,m,m'}$ be a rational section of $\overline\calL_{i,m}-\overline\calL_{i,m'}$ which restricts to $1$ on $X$. For every $\epsilon>0$ there exists $m_0$ such that when $m'\ge m\ge m_0$, both
\[
\epsilon\overline\pi^*\cD\pm\widehat\div\left(\ell_{i,m,m'}\right)
\]
are strictly effective for every $0\le i\le n$.

By linearity, we have
\begin{multline*}
\langle\overline\calL_{0,m},\cdots,\overline\calL_{n,m}\rangle(\cX_m/\cY_m)-\langle\overline\calL_{0,m'},\cdots,\overline\calL_{n,m'}\rangle(\cX_{m'}/\cY_{m'})\\
=\sum_{i=0}^n\left\langle\overline\calL_{0,m},\cdots,\overline\calL_{i-1,m},\left(\overline\calL_{i,m}-\overline\calL_{i,m'}\right),\overline\calL_{i+1,m'},\cdots,\overline\calL_{n,m'}\right\rangle(\cX_{m'}/\cY_{m'}).
\end{multline*}

Consider the Deligne products on the right-hand side for a fixed $0\le i\le n$. These are generated Zariski-locally by rational sections 
\[
\langle\ell_{0,m},\dots,\ell_{i-1,m},\ell_{i,m,m'},\ell_{i+1,m'},\dots,\ell_{n,m'}\rangle
\]
 chosen so that 
\[
\left(\bigcap_{j<i}\div(\ell_{j,m})\right)\cdot\div(\ell_{i,m,m'})\cdot\left(\bigcap_{j>i}\div(\ell_{j,m'})\right)=\emptyset.
\]
Since $\epsilon\overline\cD\pm\widehat\div\left(\ell_{i,m,m'}\right)$ is strictly effective and each $\overline\calL_{i,m},\overline\calL_{i,m'}$ is nef, for $m,m'\ge m_0$,
\begin{multline*}
\sum_{i=0}^n\left\langle\overline\calL_{0,m},\cdots,\overline\calL_{i-1,m},\left(\overline\calL_{i,m}-\overline\calL_{i,m'}\right),\overline\calL_{i+1,m'},\cdots,\overline\calL_{n,m'}\right\rangle(\cX_{m'}/\cY_{m'})
\\\le
\sum_{i=o}^n\left\langle\overline\calL_{0,m},\cdots,\overline\calL_{i-1,m},\epsilon\pi^*\cO(\overline\cD),\overline\calL_{i+1,m'},\cdots,\overline\calL_{n,m'}\right\rangle(\cX_{m'}/\cY_{m'})
\\=
\epsilon\left(\sum_{i=0}^n\deg\left(\overline\calL_{0,K}\cdots\overline\calL_{i-1,K}\cdot\overline\calL_{i+1,K}\cdots\overline\calL_{n,K}\right)\right)\cO(\overline\cD)
\end{multline*}

Thus
\[
\left\langle\overline\calL_0,\dots,\overline\calL_n\right\rangle(X/Y):=\lim_{m\to\infty}\left\langle\overline\calL_{0,m},\dots,\overline\calL_{n,m}\right\rangle(\cX_{m}/\cY_{m})
\]
is well defined on $\widehat\Pic(\cW)_{\Int}$, as we've shown that this sequence is Cauchy.
\end{proof}

Suppose we have a sequence of projective, flat morphisms 
\[
X\overset{\pi}\longrightarrow Y\overset{\rho}\longrightarrow Z\]
of projective $K$-varieties, such that $X/Y$ has pure relative dimension $n$ and $Y/Z$ has pure relative dimension $m$. Let $\oL_0,\dots,\oL_n\in\widehat\Pic(X)_{\Int}$ and $\oH_1,\dots,\oH_m\in\widehat\Pic(Y)_{\Int}$. By the projection formula for the Deligne product
\[
\left\langle\left\langle\oL_0,\dots,\oL_n\right\rangle(X/Y),\oH_1,\dots,\oH_m\right\rangle(Y/Z)=
\left\langle\oL_0,\dots,\oL_n,\pi^*\oH_1,\dots,\pi^*\oH_m\right\rangle(X/Z).
\]

When the relative context is clear, for example when $X$ is a projective $K$-variety of dimension $n$, we will often simply write
\[
\oL_0\cdots\oL_n:=\left\langle\oL_0,\dots,\oL_n\right\rangle(X/\Spec K).
\]

\subsection{$\R$-valued intersections}
Now suppose $\cB$ is a projective arithmetic model for $K$. We have an intersection product 
\[
\widehat\Pic(\cB)^{d+1}\longrightarrow\R,
\]
defined classically in the base $k$ case and by Gillet-Soul\'e~\cite{gilletsoule,gilletsoule2}, and further generalized by Zhang~\cite{Z95} and Moriwaki~\cite{Moriwaki}, in the base $\Z$ case. As with the Deligne pairing, this extends to limits, provided we only consider integrable line bundles.

\begin{lemma}
The intersection product extends to a unique continuous multilinear homomorphism $\widehat\Pic(K)_{\Int}^{d+1}\to\R$.
\end{lemma}

\begin{proof}
We proceed similarly to the previous lemma.
By the definition of $\widehat\Pic(K)_{\Int}$, it suffices to prove that the product extends to $\widehat\Pic(\cV)$, for a fixed open arithmetic model $\cV$ for $K$, and to suppose that $\overline\calL_0,\dots,\overline\calL_d\in\widehat\Pic(\cV)_{\Int}$ are all nef, and represented by sequences $\{\overline\calL_{i,m}\}_{m\ge0}$ of nef (Hermitian) line bundles on projective models $\cB_m$ for $\cU$. Fix a strictly effective $\overline\cD$ on $\cB_0$ with support on $\cB_0\backslash\cV$, and fix $\epsilon>0$. Then there exists $m_0$ such that 
\[
\epsilon\overline\cD\pm\widehat\div(\ell_{i,m,m'})
\]
is strictly effective for all $0\le i\le d$ when $m'\ge m\ge m_0$.

Let $I$ be a subset of $\{0,d\}$. We show that for fixed $I$, the sequence $\{\alpha_{I,m}\}_{m\ge0}$ defined by 
\[
\alpha_{I,m}:=\overline\cD^{d+1-|I|}\prod_{i\in I}\ocL_{i,m}
\]
converges, by induction on $|I|$. The case $I=\emptyset$ is clear, as the sequence is constant, and the lemma is proven when $|I|$ is maximal. Fix some $I$ and $m'\ge m\ge m_0$, and compute as follows:
\begin{multline*}
\left|\alpha_{I,m}-\alpha_{I,m'}\right|=
\left|\overline\cD^{d+1-|I|}\cdot\prod_{i\in I}\overline\calL_{i,m}-\overline\cD^{d+1-|I|}\cdot\prod_{i\in I}\overline\calL_{i,m'}\right|\\
=\left|\overline\cD^{d+1-|I|}\cdot\sum_{j\in I}\left(\prod_{i\in I,i<j}\overline\calL_{i,m}\cdot\left(\overline\calL_{j,m}-\overline\calL_{j,m'}\right)\cdot\prod_{i\in I,i>j}\overline\calL_{i,m'}\right)\right|
\end{multline*}
Since every $\overline\calL_{i,m},\overline\calL_{i,m'}$ is nef, this is bounded by
\begin{multline*}
\epsilon\left|\overline\cD^{d+2-|I|}\cdot\sum_{j\in I}\left(\prod_{i\in I,i<j}\overline\calL_{i,m}\cdot\prod_{i\in I,i>j}\overline\calL_{i,m'}\right)\right|\\
\le
\epsilon\left|\overline\cD^{d+2-|I|}\cdot\sum_{j\in I}\left(\prod_{i\in I,i<j}\overline\calL_{i,m}\cdot\prod_{i\in I,i>j}\overline\calL_{i,m}+\epsilon\overline\cD\right)\right|.
\end{multline*}
This last term is a linear combination of terms $\alpha_{I',m}$, with $|I'|<|I|$, and thus by induction, $\alpha_{I,m}$ is Cauchy.
\end{proof}

Fix a projective model $\pi:\cX\to\cB$. Let $\overline\calL_0,\dots,\overline\calL_n\in\widehat\Pic(\cX)$ and $\overline\cH_1,\dots,\overline\cH_d\in\widehat\Pic(\cB)$
By the projection formula for the Deligne product, 
\[
\left(\overline\calL_0\cdots\overline\calL_n\right)\cdot\overline\cH_1\cdots\overline\cH_d=\overline\calL_0\cdots\overline\calL_n\cdot\pi^*\overline\cH_1\cdots\pi^*\overline\cH_d\in\R,
\]
where the right-hand side is the classical intersection product on $\cX/k$ in the base $k$ case, or the Arakelov intersection product on $\cX/\Z$ in the base $\Z$ case. 

With both intersection products now fully defined, recall the definitions of \emph{pseudo-effective}, \emph{numerically trivial}, \emph{arithmetically positive}, and \emph{$\oL$-bounded} from Definition~\ref{positivitydefs}. Additionally, note that an \emph{arithmetically positive} adelic line bundle is always \emph{nef}, and a product of nef adelic line bundles is always $\ge0$ (\emph{pseudo-effective}).

\section{Flat and admissible metrics}\label{admissible}




In this section we construct canonical metrics for specific subsets of $\Pic(X)_{\Q}$. These are all generalized by the concept of an \emph{admissible metric} in Theorem~\ref{admissibletheorem}, however the proof of that result will require the Hodge Index Theorem, Theorem~\ref{hodgeindex}, in addition to the results of this section.

\subsection{Tate's limiting argument}\label{tateslimiting}
Let $(X,f,L)$ be a polarized dynamical system, i.e. $X$ is a geometrically normal projective $K$-variety, $f:X\to X$ is an endomorphism, $L$ is an ample line bundle on $X$, and $f^*L\isom qL$ for some $q>1$. 

\begin{lemma}\label{tateslimiting}
There exists an adelic line bundle $\oL_f$ extending $L$ in the sense that $L_f:=(\oL_{f})_K= L$, which is both integrable and nef, and such that $f^*\oL_f\cong q\oL_f$.
\end{lemma}

\begin{proof}
Fix a projective arithmetic model $\pi:\cX\to\cB$ for $X/K$, and any nef $\ocL\in\widehat\Pic(\cX)$ whose generic fiber is $L$. There exists an open subscheme $\cV\subset\cB$ such that $\cX_{\cV}\to\cV$ is flat, and over which $f$ can be extended to a morphism $f_{\cV}:\cX_{\cV}\to\cX_{\cV}$ such that $f^*_{\cV}\calL_{\cV}=q\calL_{\cV}$. By possibly enlarging $\cV$, we may assume $\cD:=\cB\backslash\cV$ is an effective Cartier divisor on $\cB$.

We have a composition of maps given by iterates of $f$,
\[
X\xrightarrow{f^m}X\hooklongrightarrow\cX,
\]
whose normalizations produce maps we denote $f_m:\cX_m\to\cX$ with a map $\pi_m:\cX_m\to\cB$. Define 
 \[
 \ocL_m:=q^{-m}f^*_m\ocL\in\widehat\Pic(\cX_m)_{\Q}.
 \]
Let $\cX_0:=\cX$ and $\ocL_0:=\ocL$. By construction, $\cX_{m,\cV}=\cX_{\cV}$ and $\calL_{m,\cV}=\calL_{\cV}$ for every $m\ge0$. The sequence $\oL_f=\{(\cX_m,\ocL_m)\}$ will be our desired adelic line bundle, once we show it is convergent. 

For each $m\ge0$, let $\widetilde\pi_m:\widetilde\cX_m\to\cB$ be a projective arithmetic model of $X/K$ which dominates both $\cX_m$ and $\cX_{m+1}$ via birational morphisms ${\phi_m:\widetilde\cX_m\to \cX_m}$ and ${\psi_m:\widetilde\cX_m\to\cX_{m+1}}$. We further assume that each $\widetilde\cX_m$ dominates $\widetilde\cX_0$ via a birational morphism $\tau_m:\widetilde\cX_m\to\widetilde\cX_0$. As above, $\phi_{m,\cV}$ and $\psi_{m,\cV}$ are both isomorphisms. 

By construction, 
\[
\phi^*_m\ocL_m-\psi^*_m\ocL_{m+1}=\frac1{q^m}\tau^*_m\left(\phi^*_0\ocL_0-\psi^*_0\ocL_{1}\right).
\]
Since $\phi^*_0\ocL_0-\psi^*_0\ocL_{1}$ is supported on $\widetilde\pi_0^*|\cD|$ and $q>1$, the sequence 
$\{\ocL_m\}$ is Cauchy,
and thus $\oL_f$ converges in $\widehat\Pic(X)_{\Cont}$. Since every $\ocL_m$ is nef, the limit is nef (and thus integrable).

\end{proof}




\subsection{The $K/k$-trace}\label{trace}

For now, consider only the base $k$ case. We give a brief overview of Chow's $K/k$-trace and image functors. Proofs can be found in~\cite{langab} and~\cite{conradtrace}. Let $A$ be an abelian variety defined over $K$. The $K/k$-trace of $A$ consists of an abelian variety $\Tr_{K/k}(A)$ defined over $k$, along with a homomorphism ${\tau:\Tr_{K/k}(A)_K\to A}$. This satisfies the universal property that any homomorphism $B_K\to A$, where $B$ is an abelian variety defined over $k$, must factor through $\tau$. In characteristic zero, $\tau$ is injective. In positive characteristic, however, the kernel of $\tau$ is finite and connected, but may be nontrivial.

The $K/k$-image is dual to the trace construction, and consists of an abelian $k$-variety $\Im_{K/k}(A)$ and homomorphism ${\lambda:A\to\Im_{K/k}(A)_K}$. It has the same universal property as the trace, but with arrows reversed. $\lambda\circ\tau$ is an isogeny and $\Tr_{K/k}(A^{\vee})=\Im(A)^{\vee}$. Additionally, any homomorphism $f:A\to B$ of abelian varieties descends to a $k$-homomorphism 
\[
f:\Tr_{K/k}(A)\to\Tr_{K/k}(B).
\]

One can think of the $K/k$-trace and image of $A$ as the largest subgroup and quotient of $A$ defined over $k$. This is true for the group $A(K)$ of $K$ points, but as noted above, true only up to an infinitesimal kernel at the level of varieties in positive characteristic.

We now use the trace to construct a special class of adelic line bundles. As before, $X$ is a projective, geometrically normal $K$ variety.
Extending $K$ if necessary, assume there exists a point $x_0\in X(K)$. We fix the following notation throughout: $\Pic_X$ is the Picard scheme of $X$, and $\Pic^0_{\Red,X}$ is its reduced neutral component, which is an abelian variety~\cite{Kleiman}
. Then $\Pic(X)$ and $\Pic^0(X)$ are the abelian groups consisting of $K$ points of $\Pic_X$ and $\Pic_{\Red,X}^0$, respectively. Define $\Alb_X$, the Albanese variety, as the dual to $\Pic^0_{\Red,X}$, and we have a unique \emph{Albanese morphism}
\[
\iota:X\to\Alb_X
\]
which takes $x_0$ to $0$, and such that any other map from $X$ to an abelian variety taking $x_0$ to $0$ must factor through $\iota$.

Fix a projective model $\cB\to k$ for $K$. We have a map
\[
\Pic\left(\Im_{K/k}(\Alb_X)\right)\longrightarrow\Pic\left(\Im_{K/k}(\Alb_X)\times_k\cB\right)
\]
via the pullback of the projection onto the first factor. But now $\Im_{K/k}(\Alb_X)\times_k\cB\to\cB$ is a projective model for $\Im_{K/k}(\Alb_X)_K$. Composing with the previous then taking the direct limit over all projective models $\cB$ for $K$ induces a map
\[
\Pic\left(\Im_{K/k}(\Alb_X)\right)\longrightarrow\widehat\Pic\left(\Im_{K/k}(\Alb_X)_K\right)_{\Mod}.
\]
This further maps into $\widehat\Pic(X)_{\Mod}$ via the pullback of the Albanese morphism. Finally, by the duality of the trace and image we can define
\[
\Tr_{K/k}\Pic^0(X):=\Tr_{K/k}\left(\Pic^0_{\Red,X}\right)(k)=\Pic^0\left(\Im_{K/k}(\Alb_X)\right)
\] 
and we've produced a morphism
\[
\widehat\tau:\Tr_{K/k}\Pic^0(X)\longrightarrow\widehat\Pic(X)_{\Int}
\]
Write ${\widehat\Tr_{K/k}\Pic^0(X)\subset\widehat\Pic(X)_{\Int}}$ to mean the image of $\widehat\tau$. If we drop the metric, $\widehat\tau$ is simply the trace map $\tau:\Tr_{K/k}(\Pic^0_{\Red,X})_K\to\Pic^0(X)$ on $K$-points.

In the base $\Z$ setting, no traces show up, as there is no field underlying $\Spec\Z$ over which to take a trace. So that we can continue to write single statements for both the base $k$ and base $\Z$ cases, however, we will use the trace notation with the stipulation that ${\widehat\Tr_{K/\Q}\Pic^0(X):=0}$ in the base $\Z$ setting.




\subsection{Flat Metrics}\label{flatsection}
Next we construct \emph{flat} adelic line bundles, which form a subgroup of $\widehat\Pic(X)_{\Int}$ orthogonal to $\widehat\Pic(X)_{\Vrt}$.
\begin{definition}
An adelic line bundle $\overline M\in\widehat\Pic(X)_{\Int}$ is called \emph{flat} provided that for any $\overline N\in\widehat\Pic(X)_{\Vrt}$ and $\overline L_2,\dots\overline L_n\in\widehat\Pic(X)_{\Int}$, we have
\[
\overline M\cdot\overline N\cdot\overline L_2\cdots\overline L_n\equiv 0.
\] 
\end{definition}

Observe that $\widehat\Tr_{K/k}\Pic^0(X)$ is numerically trivial, and thus in particular \emph{flat}: if $\cX\to\cB$ is any projective model for $X\to K$, then the induced element of $\Pic\left(\Im_{K/k}(\Alb_X)\times_k\cB\right)$ is algebraically trivial on every fiber.

\begin{lemma}\label{flatexistence}
Any $M\in\Pic^0(X)_{\Q}$ has a flat metric $\overline M$, which is unique up to $\pi^*\widehat\Pic(K)_{\Int}$.
\end{lemma}

\begin{proof}

Let $P$ be the Poincar\'e universal line bundle on $\Alb_X\times \Alb_X^{\vee}$. We also write $P$ to mean its pullback via the Albanese morphism to $X\times\Alb_X^{\vee}$. Then $P$ is determined by the property that $P|_{\{x_0\}\times \Alb_X^{\vee}}$ is trivial 
and $P|_{X\times\{\alpha\}}$ is isomorphic to the line bundle represented by $\alpha$ for every $\alpha\in\Alb_X^{\vee}(\overline K)=\Pic^0_{\Red,X}(\overline K)$.

We have projection maps $\pr_1:X\times \Alb_X^{\vee}\to X$ and $\pr_2:X\times \Alb_X^{\vee}\to \Alb_X^{\vee}$. We also have maps 
\[
[\Id_X\times m]:X\times\Alb_X^{\vee}\longrightarrow X\times\Alb_X^{\vee}
\]
given by the identity on the first component and scalar multiplication by $m$ on the second. On $\Alb_X\times\Alb_X^{\vee}$, note that $P$ is symmetric, meaning $[-1]^*P=P$, and $[2]^*P=4P$. Then on $X\times\Alb_X$, $P$ is anti-symmetric in the sense that ${[\Id_X\times -1]^*P\cong -P}$, and thus $[\Id_X\times 2]^*P\cong 2P$. By Tate's limiting argument, Lemma~\ref{tateslimiting}, $P$ has a unique metric $\overline P \in\widehat\Pic(X\times\Alb_X^{\vee})_{\Int}$ such that 
\[
[\Id_X\times 2]^*\overline P\cong2\overline P.
\]
Assume that $M\in\Pic^0(X)$; by linearity the result will hold for $\Pic^0(X)_{\Q}$. Let $\alpha\in\Alb_X^{\vee}(\overline K)$ correspond to the line bundle $M$. We claim that 
\[
\overline M:=\overline P|_{X\times\{\alpha\}}\in\widehat\Pic(X)_{\Int}
\]
is flat. Let $\overline N\in \widehat\Pic(X)_{\Vrt}$ and $\oL_2,\dots,\oL_n\in\widehat\Pic(X)_{\Int}$. Consider the adelic line bundle
\[
\overline Q:=\left\langle \overline P,\pr_1^*\overline N,\pr_1^*\oL_2,\dots,\pr_1^*\oL_n\right\rangle\left((X\times\Alb_X^{\vee})\big/\Alb_X^{\vee}\right)\in\widehat\Pic(\Alb_X^{\vee})_{\Int}
\] 
given by the relative intersection product. The underlying line bundle $Q$ is trivial since $N=\cO_X$, and $[2]^*\overline Q\cong 2\overline Q$, thus $\overline Q$ itself must be trivial. But then by the projection formula,
\[
\left\langle \overline M,\overline N,\oL_2,\dots,\oL_n\right\rangle\left(X\big/\Spec(K)\right)\in\widehat\Pic(K)_{\Int}
\]
must also be zero.

All that remains to show is that $\overline M$ is unique up to $\pi^*\overline\Pic(K)_{\Int}$. Suppose $\overline M'$ is also flat, and $M=M'$. Then $\overline N:=\overline M-\overline M'$ is vertical. Thus
\[
\overline N^2\cdot\oL_2\cdots\oL_n=\left(\overline M-\overline M'\right)^2\cdot\oL_2\cdots\oL_n\equiv 0.
\]
By the vertical case of the Hodge Index Theorem (Section~\ref{verticalcase}), which doesn't require any of the results of this section, $\overline N\in\pi^*\widehat\Pic(K)_{\Int}.$
\end{proof}




\section{Heights}\label{heightssection}

As usual let $K/k$ be any finitely generated extension of fields, and let $d$ (resp. $d+1$) to be the transcendence degree of $K$ over $k$ in the base $\Z$ (resp. base $k$) setting. Let $X$ be a geometrically normal projective variety of dimension $n$ over $K$. Let $\overline L\in\widehat\Pic(X)_{\Int}$ such that $L$ is ample. We define a vector-valued height function, which takes values in the $\Q$-vector space $\widehat\Pic(K)_{\Int}$.

\begin{definition}

let $Z$ be any closed $\overline K$-subvariety of $X$. Write $\widetilde Z$ for the minimal $K$-subvariety of $X$ which contains $Z$.
The \emph{height} of $Z$ with respect to $\overline L$ is 
\[
h_{\overline L}(Z):=\frac{\left(\overline L|_{\widetilde Z}\right)^{\dim Z+1}}{\left(\dim Z+1\right)\left(L|_{\widetilde Z}\right)^{\dim Z}}\in\widehat\Pic(K)_{\Int}
\]
In particular, this defines a height $X(\overline K)\to\widehat\Pic(K)_{\Int}$ on points, and for this height we may drop the requirement that $L$ is ample.
\end{definition}

\begin{remark}

When $d=0$, so that $K$ is either a number field or a transcendence degree one function field over $k$, we have an arithmetic degree map $\widehat\Pic(K)_{\Int}\to\R$, and vector-valued heights composed with the degree map reproduce the usual $\R$-valued height functions.
\end{remark}

When we work with polarized dynamical systems, a particular class of heights, \emph{canonical heights} will be especially useful. As in Section~\ref{admissible}, let $f:X\to X$ be a dynamical system polarized by an ample line bundle $L\in\Pic(X)$ such that $f^*L\cong qL$, so that Lemma~\ref{tateslimiting} produces a canonical metric $\oL_f$.

\begin{definition}
The height function $h_f:=h_{\oL_f}$
is called the \emph{canonical height} for the polarized dynamical system $f:X\to X$. It has the property that $h_f(f(x))=qh_f(x)$.
\end{definition}

On an abelian $K$-variety $A$, let $[2]$ be scalar multiplication by 2, and let $L\in\Pic(A)$ be any ample, symmetric line bundle. Then $[2]^*L\cong 4L$. This produces the canonical height $h_{\oL}=h_{\oL_{[2]}}$ on $A(\overline K)$.  

We name a particular canonical height, the \emph{N\'eron-Tate height}, on $A$. Let $\theta$ be the $\theta$-divisor on $A$. Then the line bundle $\Theta=\cO(\theta)+[-1]^*\cO(\theta)$ is symmetric and ample. 
\begin{definition}	
The vector-valued \emph{N\'eron-Tate height}, $h_{\Nt}$, is defined to be
\[
h_{\Nt}:=\frac12h_{\overline\Theta}:X(\overline K)\longrightarrow\widehat\Pic(K)_{\Int}.
\]

\end{definition}

For general heights, since $h_{\overline L}(Z)\in\widehat\Pic(K)_{\Int}$, we can produce values in $\R$ by choosing $\overline H_1,\dots,\overline H_d\in\widehat\Pic(K)_{\Int}.$ In the base $\Z$ case, this produces a Moriwaki height, and the choice corresponds to a choice of polarization~\cite{moriwakideligne,moriwakifg1}. In the base $k$ setting this is the same choice required to specify a function field height, or equivalently, a set of normalized absolute values on $K$ which satisfy the product formula~\cite{conradtrace,langdiophantinegeometry,serremordellweil}. Note that these sources fix a single model $\cB$ for $K$, but the general case follows by approximation. More specifically, define

\begin{definition}
\[
h_{\overline L}^{\overline H_1,\dots,\overline H_d}(Z):=h_{\oL}(Z)\cdot\oH_1\cdots\oH_d=\frac{\left(\overline L|_{\widetilde Z}\right)^{\dim Z+1}\cdot \overline H_1\cdots\overline H_d}{\left(\dim Z+1\right)\left(L|_{\widetilde Z}\right)^{\dim Z}}\in\R.
\]
\end{definition}

We list several properties of height functions, all of which are either well-known (see~\cite{langdiophantinegeometry,serremordellweil}, for example) or come from Moriwaki~\cite{moriwakifg1}.
\begin{proposition}
Let $\oL,\oM\in\widehat\Pic(X)_{\Int}$ and $\oH_1,\dots,\oH_d\in\widehat\Pic(K)_{\Int}$.
\begin{enumerate}
\item
If $L\in\Pic(X)_{\Q}$ is ample and $\oL$ is nef, then $h_{\oL_f}(Z)\ge0$ for all subvarieties $Z$ of $X$.
\item
$h_{\oL}(x)+h_{\oM}(x)=h_{\oL+\oM}(x)$ for all $x\in X(\overline K)$.
\item
If $L=M$, then $h_{\oL}^{\oH_1,\dots,\oH_d}$ and $h_{\oM}^{\oH_1,\dots,\oH_d}$ are equal up to $O(1)$. In particular, if $M=\cO_X$, then $h_{\oM}^{\oH_1,\dots,\oH_d}=O(1)$.
\item
If $L$ is ample and $M$ is arbitrary, and $\oH_1,\dots,\oH_d$ are nef, then there exists a positive real constant $C>0$ such that 
\[
\left|h_{\oM}^{\oH_1,\dots,\oH_d}\right|\le C\cdot h_{\oL}^{\oH_1,\dots,\oH_d}+O(1).
\]
In particular, by (3), $h_{\oL}^{\oH_1,\dots,\oH_d}$ is bounded below.
\end{enumerate}
\end{proposition}

\subsection{Heights in base $\Z$}
Throughout this subsection, assume we are in the base $\Z$ setting. These results are stated in~\cite{yz2}. The height $h_{\overline L}^{\overline H_1,\dots,\overline H_d}(Z)$ is defined in~\cite{moriwakifg1} when $\overline L$ and $\overline H_1,\dots\overline H_d$ are all defined on the same projective arithmetic model $\cX\to\cB$, and generalized to limits of such in~\cite{moriwakifg2}. 
As a slight generalization of Moriwaki's work, we have the following.

\begin{theorem}[Northcott property for vector-valued heights, base $\Z$ case]\label{NorthcottZ}
Let $\oL\in\widehat\Pic(X)_{\Int}$ and suppose $L$ is ample. Let $D\in\Z_{>0}$ be any positive bound and $\alpha\in\widehat\Pic(K)_{\Int}$. Then the set
\[
\left\{x\in X(\overline K)\;:\;[K(x):K]\le D,\;h_{\oL}(x)\le\alpha\right\}
\]
is finite. 
\end{theorem}
Recall that $h_{\oL}(x)\le\alpha$ means $\alpha-h_{\oL}(X)\in\widehat\Pic(K)_{\Int}$ is pseudo-effective. As is well known for heights over number fields, we have the following corollary.

\begin{cor}\label{heightprepcorZ}
Let $f:X\to X$ be a polarized dynamical system. Then we have an equality
\[
\Prep(f)=\{x\in X(\overline K)|h_f(x)\equiv0\}
\]
of $f$-preperiodic points and canonical height zero points. If $A$ is an abelian variety, then $h_{\NT}(x)\equiv0$ if and only if $x\in A(\overline K)$ is torsion. 
\end{cor}
The corollary follows immediately from the fact that 
\[
h_f(f^m(x))=q^mh_f(x)
\]
for all $x\in X(\overline K)$. We now prove the theorem.

\begin{proof}
Moriwaki proves the following:
\begin{lemma}[Theorem 4.3 of~\cite{moriwakifg1}]
Let $K$ be a finitely  generated extension of $\Q$. Suppose $\oL\in\widehat\Pic(X)_{\Int}$, that $L$ is ample, that $\overline H_1,\dots,\overline H_d\in\widehat\Pic(K)_{\Int}$ all come from a single projective arithmetic model $\cX\to\cB$ for $X/K$, and that $L$ is ample and $\overline H_1,\dots,\overline H_d$ are all nef and big. Then for any $\alpha\in\R$ and $D\in\Z_{>0}$, the set
\[
\left\{x\in X(\overline K)\big| h_{\overline L}^{\overline H_1,\dots,\overline H_d}(x)\le \alpha, [K(x):K]\le D\right\}
\]
is finite.
\end{lemma}
This proves the theorem immediately for any $\oL$ coming from a single projective model $\cX$ (as opposed to a limit). For general $\oL$, chose $\oL_0$ such that $\oL_0$ is defined on a single projective model and $L=L_0$. Then the difference
\[
h_{\overline L}^{\overline H_1,\dots,\overline H_d}(x)-h_{\overline L_0}^{\overline H_1,\dots,\overline H_d}(x)
\]
is a bounded function.
\end{proof}

\subsection{Heights in base $k$}\label{heightsbaseksubsection}
Now assume we are in the base $k$ setting. 
The results of the previous subsection hold when $k$ is a finite field, but fail otherwise. Baker~\cite{bakerisotrivial} notes that even for $\phi:\P^1\to\P^1$, if $k$ is infinite, the set
\[
\left\{x\in\P^1(K):h_{\phi}(x)\le M\right\}
\]
is infinite when $M$ is large enough. 

We do get results, however, if we study sets of height zero instead of sets of bounded height. Since many of the properties of general heights only hold up to a bounded function, we also restrict to canonical heights. Additionally, we must impose necessary non-isotriviality conditions. Consider the example of $k$ algebraically closed, and a polarized dynamical system $f_K:X_K\to X_K$, where $X$ is a projective variety defined over $k$. Then $X$ has infinitely many $k$ points, and all of these have canonical height zero.

Recall the definitions of a \emph{constructible map}, \emph{constructibly isotrivial}, and \emph{totally non-isotrivial} from Definition~\ref{totallynonisotrivial}.
When $X$ has dimension 1, constructibly isotrivial is the same as being birationally isomorphic to a curve defined over $k$, since a constructible isomorphism can be made into a birational isomorphism via a Frobenius twist. In higher dimension and positive characteristic, however, constructibly isotrivial is a weaker notion than \emph{isotrivial}, meaning isomorphic to a variety defined over $k$, in the category of morphisms of varieties over $\overline K$. In fact, we can see from abelian varieties why \emph{constructibly isotrivial} and not \emph{isotrivial} is the right definition to use. Let $A/K$ be an abelian variety. In positive characteristic, the trace map
\[
\tau_{K/k}:\Tr(A)_K\to A
\]
may have a nontrivial kernel, which is a finite connected $K$-group scheme not defined over $k$. See~\cite[Example 4.4]{conradtrace} for an explicit example. When this kernel is nontrivial, $\Tr(A)_K$ is isotrivial, but its image in $A$ is not. $\Tr(A)_K$ and its image are constructibly isomorphic but not isomorphic. But the N\'eron-Tate height is nonetheless zero on the entire image of $\Tr(A)_K$ in $A$.

\begin{theorem}[Northcott property for vector-valued heights, base $k$ case]
Suppose $f:X\to X$ is totally non-isotrivial, or that $k$ is a finite field. Then for any $D\in\Z_{>0}$, the set
\[
\left\{x\in X(\overline K)\;:\;h_f(x)\equiv0,\; [K(x):K]\le D\right\}
\]
is finite.
\end{theorem}

As in the base $\Z$ case, this has an immediate implication for preperiodic points and torsion points.

\begin{cor}\label{heightprepcork}
When $(X,f)$ is totally non-isotrivial, or when $k$ is a finite field,
\[
\Prep(f)=\{x\in X(\overline K)|h_f(x)\equiv 0\}.
\]
If $A$ is an abelian variety, $h_{\NT}(x)\equiv0$ if and only if $x\in A(K)$ is in 
\[
A(\overline K)_{\Tor}+\Tr_{K/k}(A)(\overline k).
\]

\end{cor}

This generalizes a theorem of Baker~\cite{bakerisotrivial} for rational functions on $\P^1$ (also proven by Benedetto~\cite{benedetto} for polynomials), and the Lang-N\'eron theorem~\cite{langneron} for abelian varieties. We prove the theorem.

\begin{proof}
We use the following result of Chatzidakis and Hrushovski~\cite[Theorem 1.11]{chatzidakis1}, stated more directly in the context of polarized endomorphisms in~\cite[Remark 5.6]{ChatzidakisICM}:
\begin{lemma}
Let $K/k$ be a finitely generated regular field extension, and $f:X\to X$ a polarized endomorphism of projective varieties. Then if a limited set of $X(\overline K)$ is dense, $(X,f)$ is constructibly isotrivial.
\end{lemma}
Fix a model $\cB$ for $K$ and pick any ample 
$\cH_1,\dots,\cH_d\in \Pic(\cB)_{\Int}\subset\widehat\Pic(K)_{\Int}$. 
Then
\begin{multline*}
S:=\left\{x\in X(\overline K):h_f(x)\equiv0, [K(x):K]\le D\right\}\\
\subset\left\{x\in X(\overline K):h_f^{\cH_1,\dots,\cH_d}(x)=0, [K(x):K]\le D\right\},
\end{multline*}
and by ~\cite[Theorem 4.7]{chatzidakis1}, the right hand side is a limited set. If $S$ is not finite, its Zariski closure, $\overline S$, is a positive dimensional, possibly reducible closed subvariety of $X$. Because $f(S)\subset S$, we have $f(\overline S)\subset\overline S$, and thus $\overline S$ contains a periodic, irreducible component $Y$. We thus have a positive dimensional sub-dynamical system $f^m:Y\to Y$, in which the set of canonical height zero points is dense. Thus by the lemma, $(Y,f^m)$ is constructibly isotrivial.

\end{proof}

While, as noted previously, the Northcott property fails in base $k$, we expect that the following height gap statement should hold:

\begin{conj}
Suppose $(X,f)$ is totally non-isotrivial. Choose some transcendence degree one intermediate field $k_1$ as usual. Then there exists some $\alpha\in\widehat\Pic(k_1)_{\Int}$ with positive arithmetic degree such that the set
\[
\left\{x\in X(\overline K)\;:\;h_f(x)\le\pi^*\alpha,\; [K(x):K]\le D\right\}
\]
is finite.
\end{conj}
This is proven by Baker~\cite{bakerisotrivial} for rational functions on $\P^1$, and follows from the Lang-N\'eron Theorem (see~\cite{conradtrace}) for abelian varieties, but appears unknown in general.

\subsection{Comparing relative heights}\label{heightscompare}

Consider the example $K=\Q(t)$, and $k=\Q$. This can be viewed in base $\Z$, with $d=1$, or in base $k$, with $d=0$, and the two viewpoints carry different height functions. We now discus how to compare these.  

Thus far our heights have been defined on $X(\overline K)$ with respect to the finitely generated extension $K/k$. Suppose this extension has transcendence degree at least one in base $\Z$, or at least two in base $k$, and fix an intermediate subfield $k\subset k_1\subset K$. We further require that each subextension have positive transcendence degree, except in the case that $k_1$ is a number field, where we view $k_1/k$ as a base $\Z$ extension, and $K/k_1$ as base $k$. For the remainder of this section only, we introduce more specific notation, to avoid ambiguity: Let the three groups $\widehat\Pic(K/k)_{\Int}$, $\widehat\Pic(K/k_1)_{\Int}$, and $\widehat\Pic(k_1/k)_{\Int}$ refer to the construction in Section~\ref{defssection} for the finitely generated extensions of fields $K/k$, $K/k_1$, and $k_1/k$, respectively. Similarly, instead of writing $\widehat\Pic(X)_{\Int}$, we will consider $\widehat\Pic(X/K/k)_{\Int}$ and $\widehat\Pic(X/K/k_1)_{\Int}.$

Chose an adelic line bundle $\oL\in\widehat\Pic(X/K/k)$. This defines a height 
\[
h_{\oL}(Z)\in\widehat\Pic(K/k)_{\Int},
\]
where $Z$ is any closed $\overline K$-subvariety of $X$. Let $e$ be the transcendence degree of $k_1/k$, and choose any $\oH_1,\dots,\oH_e\in\widehat\Pic(k_1/k)_{\Int}$. By pulling back $\rho:\Spec K\to\Spec k_1$, these lift to 
\[
\rho^*\oH_1,\dots,\rho^*\oH_e\in\widehat\Pic(K/k)_{\Int}.
\]
Now consider the intersection
\[
h_{\oL}(Z)\cdot (\rho^*\oH_1)\cdots(\rho^*\oH_e).
\]
By the projection formula for the relative intersection pairing, this produces a vector in $\widehat\Pic(K/k_1)_{\Int}$. Note that on a fixed projective arithmetic model $\cB$ for $K$ with a flat projective morphism down to a projective arithmetic model $\cC$ for $k_1$, the adelic line bundles $\rho^*\oH_i$ correspond to vertical adelic line bundles in $\widehat\Pic(\cB_{k_1}/k_1/k)_{\Vrt}.$

Thus, letting $\oL_{k_1}$ be the restriction of $\oL$ to $\widehat\Pic(X/K/k_1)$, we've produced from $h_{\oL}$ a height
\[
h_{\oL_{k_1}}(Z)\in\widehat\Pic(K/k_1)_{\Int}.
\]

This height carries less information than the original height all the way down to $k$. For example, if $f:X\to X$ is a dynamical system, there could exist a periodic subvariety defined over $k_1$ but not over $k$, so that $X$ may be totally non-isotrivial over $k$, but not over $k_1$. Thus the canonical height over $k$ is granular enough for the Northcott property to hold, but the canonical height over $k_1$ is not. On the other hand, if $X$ is totally non-isotrivial over $k_1$, the height over $k_1$ may be easier to work with, and thus preferable. 




\section{The essential minimum}\label{essentialminimumsec}
Having established the main properties of vector-valued heights on $X(\overline K)$, we now relate the height of a variety to the heights of the points it contains. This extends Zhang's \emph{essential minimum} result~\cite[Theorem 1.11]{Z95} and subsequent work of Gubler~\cite[Theorem 4.1]{G07} and Moriwaki~\cite[Corollary 5.2]{moriwakifg1} to the setting of this paper.

\begin{definition}\label{moriwakicondition}
Let $\oH\in\widehat\Pic(K)_{\Nef}$. We say that $\oH$ satisfies the \emph{Moriwaki condition} provided that the following hold:
\begin{enumerate}
\item
$\oH$ is induced by a single (nef) $\ocH\in\widehat\Pic(\cB)$ for some projective arithmetic model $\cB$ for $K$,
\item
$\oH^{d+1}=0$, and
\item
$H_{k_1}^d>0$ for some intermediate field $k\subset k_1\subset K$ such that $k_1$ has transcendence degree one over $k$. In base $\Z$, we must have $k_1=\Q$.
\end{enumerate}
\end{definition}

\begin{remark}
Condition (3) in base $k$ is equivalent to the existence of an ample line bundle $A\in\Pic(\cB)$ such that $H^d\cdot A>0$, which avoids the need to find a specific $k_1$, however, we won't require this fact.
\end{remark}

\begin{definition}
To keep our notation more concise, we write ${h_{\oL}^{\oH}:=h_{\oL}^{\oH,\dots,\oH}}$.
Define $\lambda_1^{\oH}(X,\oL)$, the \emph{essential minimum} of $X$ with respect to $\oL$ and $\oH$, as
\[
\lambda_1^{\oH}(X,\oL):=\sup_{U\subset X}\inf_{x\in U(\overline K)}h^{\oH}_{\oL}(x),
\]
where the supremum is taken over all Zariski open $U\subset X$.
\end{definition}

\begin{theorem}\label{essentialminimum}

Suppose $\oH$ satisfies the Moriwaki condition and $\oL\in\widehat\Pic(X)_{\Nef}$, with $L$ ample on $X$. Then
\[
\lambda_1^{\oH}(X,\oL)\ge h_{\oL}^{\oH}(X).
\]
\end{theorem}

When $d=0$, this is proven in~\cite[Theorem 1.1]{Z95} in the base $\Z$ case, and in~\cite[Theorem 4.1]{G07} in base $k$. In base $\Z$, Moriwaki proves this for general $d$ in~\cite[Corollary 5.2]{moriwakifg1} and model metric $\oL$.  
Here we adapt Moriwaki's proof to both the base $\Z$ and base $k$ cases.

It suffices to prove this, as Moriwaki does, under the additional assumption that $\oL$ comes from a single model $\ocL$ on $\cX\to\cB$, since both the essential minimum and the height of $X$ are approximated uniformly as $\oL$ is approximated by such models. We first prove a weaker statement.

\begin{lemma}\label{essentialminimasublemma1}
Suppose $\calL\in\Pic(\cX)$ is relatively nef over $\cB$, that $\calL_K$ is ample on $X$, that the Hermitian metric on $\ocL$ is semipositive, and that $h_{\ocL}^{\ocH}(X)>0$. Then $\lambda_1^{\ocH}(X,\ocL)\ge0$.
\end{lemma}

\begin{proof}
By the assumption $\ocH^{d+1}=0$, we have an equality of height functions (on subvarieties of any dimension)
\[
h^{\ocH}_{\ocL}=h^{\ocH}_{\ocL+m\pi^*\ocH}
\]
for any rational number $m$. Since 
\[
\left(\ocL+m\pi^*\ocH\right)^{n+d+1}=\binom{n+d+1}{d}\ocL^{n+1}\ocH^{d}m^d+O\left(m^{d-1}\right),
\]
by replacing $\ocL$ with $\ocL+m\pi^*\ocH$ for sufficiently large positive $m$, it suffices to prove the result under the additional assumption that $\ocL^{n+d+1}>0$.

After possibly replacing $\ocL$ once more with a positive multiple, by Moriwaki's arithmetic Bertini theorem~\cite[Theorem 2.1]{moriwakifg1}, we can find a non-zero global section $s\in H^0(\cX,\calL)$ which is small in the sense that the supremum norm $||s||_{\Sup}<1$ on $X(\C)$ (we simply ignore this condition in base $k$; the equivalent result is also in~\cite[Lemma 4.2]{G07}). Let $Y=\div(s)_K\subset X$. Then for any $x\in X\backslash Y(\overline K)$,
\[
h^{\ocH}_{\ocL}(x)\ge0,
\]
since this height can be computed using the section $s$, and $\ocH$ is nef. Thus, $\lambda_1^{\ocH}(X,\ocL)\ge0$. 

\end{proof}

We now prove Theorem~\ref{essentialminimum}.

\begin{proof}
Fix $k_1$, such that $\ocH_{k_1}^d>0$. We view $\ocH$ as a metrized line bundle $(\cH_{k_1},||\cdot||)\in\widehat\Pic(\cB_{k_1})_{\Int}$. 
Define $\ocA=\ocH+\cO(F)$, where $F$ is any irreducible fiber of $\cB$ over the unique projective curve with function field $k_1$. Since $\ocH^{d+1}=0$, we have
\[
\ocA\cdot\ocH^d>0.
\]
Thus we can let $\lambda$ be any rational number such that
\[
\lambda<\frac{h_{\ocL}^{\ocH}(X)}{\ocA\cdot\ocH^{d}},
\]
and since $\ocA\in\widehat\Pic(K)_{\Int}$ and $\left(\pi^*\ocA\right)^2$ is numerically trivial on $X$, we have
\[
\left(\ocL-\lambda\pi^*\ocA\right)^{n+1}\cdot\left(\pi^*\ocH\right)^d>0,
\]
and thus also
\[
\left(\ocL+\lambda\pi^*(\ocH-\ocA)\right)^{n+1}\cdot\left(\pi^*\ocH\right)^d>0.
\]
Now we can apply Lemma~\ref{essentialminimasublemma1} to $\ocL+\lambda\pi^*(\ocH-\ocA)$, and get
\[
\lambda_1^{\ocH}\left(X, \ocL+\lambda\pi^*(\ocH-\ocA)\right)=\lambda_1^{\ocH}\left(X, \ocL-\lambda\pi^*\ocA\right)\ge0.
\]
by linearity, this means
\[
\lambda_1^{\ocH}\left(X, \ocL\right)\ge\lambda\ocA\cdot\ocH^d,
\]
and since $\lambda$ can be any rational number below $h_{\ocL}^{\ocH}(X)/(\ocA\cdot\ocH^d)$, we have
\[
\lambda_1^{\ocH}(X,\ocL)\ge h_{\ocL}^{\ocH}(X).
\]
\end{proof}

The following theorem will allow us to make use of the essential minimum without having to fix specific $\oH$ satisfying the Moriwaki condition. 

\begin{lemma}\label{moriwakinumericallytrivial}
Suppose $\oM\in\widehat\Pic(K)_{\Nef}$ has the property that $\oM\cdot\oH^d=0$ for every $\oH\in\widehat\Pic(K)_{\Nef}$ which satisfies the Moriwaki condition. Then $\oM\equiv0$.
\end{lemma}

\begin{proof}
We follow the proof set out in the base $\Z$ setting in~\cite{yz2}. Fix an open model $\cV$ for $K$, with $\oM\in\widehat\Pic(\cV)_{\Int}$. 
Let $\cW$ be an open model for $k_1$, with $\cW=\Spec\Z$ in base $\Z$. After possibly replacing $\cV$ with an open subscheme, we can assume we have a flat morphism $\cV\to\cW$, and that every closed point $x\in\cV(\overline k_1)$ corresponds to a one dimensional horizontal (over $\cW$) closed integral subscheme $\overline x$ of $\cV$. 

Fix such a closed point $x_0\in\cV(\overline k_1)$. By Noether noralization, there exists a finite morphism $\phi:\cV\to\cV_0$ to an open subset $\cV_0\subset\P^d_{\cW}$, and we may assume that the image of $x_0$ is the point $P_0=(0,\dots,0,1)\in\P^d_{k_1}$. Call $\mathcal P_0$ its closure in $\P^d_{\cW}$.

Define $\overline\cO(1)$ to be the admissible lift of $\cO(1)\in\Pic(\P^d_{k_1})$ under the square map. This corresponds to $\cO(1)$ on $\P^d_{\cW}$, with an admissible Hermitian metric in base $\Z$. Then the Moriwaki condition is satisfied, as can be seen by evaluating intersections via coordinate hyperplane sections, and then by hypothesis, $\oM\cdot\phi^*\overline\cO(1)=0$. Using the same sections, $\overline\cO(1)^d$ can be represented by the 1-cycle $\mathcal P_0$ in base $k$, and by $(\mathcal P_0,\mathfrak g_0)$ for some semi-positive $\mathfrak g_0$ on $\P^d(\C)$ in base $\Z$. We can now compute:
\[
0=\oM\cdot\phi^*\overline\cO(1)^d\ge\oM\cdot\phi^*\mathcal P_0\ge \oM|_x=h_{\oM}(x)\ge0,
\]
where the inequalities follow from the positivity of $\mathfrak g_0$ and from approximating $\oM$ by nef models on projective arithmetic models of $\cV$. Thus we've proven that the height function defined by $\oM$ is trivial on $\cV(\overline k_1)$.

Next, fix a projective model $\cB$ for $\cV$, which we may assume is flat and projective over $\Spec \Z$ or the unique projective model for $k_1$, and let $\ocM_m$ be (Hermitian, in base $\Z$) line bundles on projective models $\cB_m$ for $\cV$, which converge to $\oM$. We may assume that each $\cB_m$ dominates $\cB$, and then convergence is defined to mean that for some strictly effective arithmetic divisor $\ocD$ with support $\cB\backslash\cV$, we have 
\[
-\epsilon_m\cO(\ocD)<\ocM_m-\oM<\epsilon_m\cO(\ocD)
\]
for a sequence of positive real numbers $\{\epsilon_m\}$ converging to zero. Let $\widehat\Pic(\cV_{k_1}/k_1)$ mean the completion of the projective limit of groups $\Pic(\cB_{m,k_1})$ as $\cB_m$ varies, so that we can think of $\oM_{k_1}\in\widehat\Pic(\cV_{k_1}/k_1)$ as the generic fiber of $\oM$. If $\oM$ came from a single model $\ocM$ on $\cB$, the following would be a straightforward calculation applying the essential minimum for $d=0$ on $\widehat\Pic(\cB_{k_1})$, but for general $\oM$ we must take care to make sure this extends to the limit of such models.

Since $\oM$ defines a trivial height function and $\epsilon_m\cO(\ocD)-\ocM_m+\oM$ is effective, the height function defined by $\epsilon_m\cO(\ocD)-\ocM_m$ is non-negative on $\cV(\overline k_1)$. Let $A_1,\dots,A_{d-1}\in\Pic(\cB_{k_1})$ be very ample. By Seidenberg's Bertini's theorem, we can find sections which cut out a curve which intersects $\cV_{k_1}\subset\cB_{k_1}$ nontrivially. Since $\cV\subset\cB$ is open, the height given by $\epsilon_m\cO(\ocD)-\ocM_m$ is bounded below on this curve, hence the generic fiber $\epsilon_{m}\cO(\ocD)_{k_1}-\ocM_{m,k_1}$ is nef, and
\[
\left(\epsilon_{m}\cO(\ocD)_{k_1}-\ocM_{m,k_1}\right)\cdot A_1\cdots A_{d-1}\ge0.
\]
Taking linear combinations of the $A_i$s and letting $m\to\infty$, we see that since also each $\ocM_{m,k_1}$ is nef, $\oM_{k_1}
$ must be numerically trivial.

Now let $\ocH$ be any (Hermitian) line bundle on $\cB$ such that $\ocH_{k_1}$ is ample. By definition of $\lambda_1$, we can find a dense sequence of points $\{x_j\}\subset\cV(\overline{k_1})$ such that
\[
\lim_{j\to\infty}h_{\ocH}(x_j)=\lambda_1(\cV_{k_1},\ocH),
\]
the essential minimum of $\ocH$ on $\cV_{k_1}$. Since $\ocH_{k_1}$ is ample, there exists a positive real constant $C$ such that 
\[
h_{\cO(\ocD)}(x_j)\le C h_{\ocH}(x_j)
\]
for all $j$, and since the right hand side is convergent and in particular bounded as $j\to\infty$, this means $h_{\ocM_m}(x_j)$ converges uniformly to $h_{\ocM}(x_j)$. Then
\[
\limsup_{m\to\infty}\lambda_1\left(\cV_{m,k_1},\ocH+\ocM_m\right)\le\lambda_1\left(\cV_{k_1},\ocH\right),
\]
and since $\ocM_m$ is nef, this must be an equality.

Now apply Zhang's Theorem on essential minima,~\cite[Theorem 1.11]{Z95} when $k_1=\Q$ and~\cite[Theorem 4.1]{G07} when $k_1$ is a transcendence degree one function field over $k$, and recall that $\oM_{k_1}$ is numerically trivial to get
\[
\lambda_1\left(\cV_{k_1},\ocH\right)=\lim_{m\to\infty}\lambda_1\left(\cV_{m,k_1},\ocH+\ocM_m\right)
\ge \lim_{m\to\infty}h_{\ocH+\ocM_m}\left(\cV_{k_1}\right)
=\frac{\left(\ocH+\oM\right)^{d+1}}{(d+1)\left(\cH_{k_1}\right)^d}.
\] 
Since all of the above hypotheses hold for any positive multiple of $\oM$ and the left-most quantity doesn't depend on $\oM$, we also get that $(\ocH+a\oM)^{d+1}$ is bounded for all $a>0$, and thus we must have $\oM\cdot\ocH^d=0$. 

Finally, let $\ocH_1,\dots,\ocH_d$ be any line bundles with ample generic fibers $\ocH_{i,k_1}$. Then
\[
\oM\cdot\left(t_1\ocH_1+\cdots+t_d\ocH_d\right)^d=0
\]
for all positive real constants $t_1,\dots,t_d$. Thus
\[
\oM\cdot\ocH_1\cdots\ocH_d=0.
\]
taking linear combinations and varying $\cB$, we get $\oM\equiv0$ as desired.

\end{proof}




\section{The Hodge Index Theorem}\label{hodgeindexsection}

This section proves Theorem~\ref{hodgeindex}, beginning with the inequality, then proving the equality, first for vertical line bundles, then in the case where $X$ is a curve, and finally in general.

\subsection{Inequality}
Let $\pi:\cX\to\cB$ be a projective arithmetic model for $X/K$, let $\ocM\in\widehat\Pic(\cX)$, and let $\ocL_2,\dots,\ocL_n\in\widehat\Pic(\cX)$ be nef, with big generic fibers $\ocL_{2,K},\dots,\ocL_{n,K}\in\Pic(X)$. Let $\ocH_1,\dots,\ocH_d\in\widehat\Pic(\cB)$ be nef. Suppose
\[
\ocM_K\cdot\ocL_{2,K}\cdots\ocL_{n,K}=0.
\]
Since every $\oM,\oL_i$ can be approximated by models as above, it suffices to prove that 
\[
\ocM^2\cdot\ocL_2\cdots\ocL_n\cdot\pi^*\ocH_1\cdots\pi^*\ocH_d\le0.
\]

Let $k_1$ be $\Q$ in base $\Z$, or a subfield of $K$ with transcendence degree one over $k$ in base $k$, as in previous sections.

View $\cX$ as a model for $\cX_{k_1}$. Then
\[
\ocM_{k_1}\cdot\ocL_{2,k_1}\cdots\ocL_{n,k_1}\cdot\left(\pi^*\ocH_1\right)_{k_1}\cdots\left(\pi^*\ocH_d\right)_{k_1}=\left(M\cdot L_2\cdots L_n\right)\left(\cH_{1,k_1}\cdots\cH_{d,k_1}\right)=0.
\]
 The inequality then follows immediately from the arithmetic Hodge Index Theorem over $k_1$ in~\cite{carney,yz}.
 
As a corollary, we have the following Cauchy-Schwarz inequality:
\begin{cor}\label{cauchyschwarz}
Suppose $\oM,\oN\in\widehat\Pic(X)_{\Int}$, and $\oL_2,\dots,\oL_n\in\widehat\Pic(X)_{\Nef}$. Let $\oH_1,\dots,\oH_d\in\widehat\Pic(K)_{\Nef}$. Then
\begin{multline*}
\left(\oM\cdot\oN\cdot\oL_2\cdots\oL_n\cdot\pi^*\oH_1\cdots\pi^*\oH_d\right)^2
\\\le\left(\oM^2\cdot\oL_2\cdots\oL_n\cdot\pi^*\oH_1\cdots\pi^*\oH_d\right)\left(\oN^2\cdot\oL_2\cdots\oL_n\cdot\pi^*\oH_1\cdots\pi^*\oH_d\right).
\end{multline*}
\end{cor}
Using the previously established inequality, the proof is the same as that for the usual Cauchy-Schwarz inequality.

\subsection{Equality for vertical adelic line bundles}\label{verticalcase}
We now prove the equality part of Theorem~\ref{hodgeindex} under the assumption that $\oM\in\widehat\Pic(X)_{\Vrt}$. The theorem additionally imposes the hypotheses that $\oL_i\gg0$, that $\oM$ is $\oL_i$-bounded for every $2\le i\le n$, and that
\[
\oM^2\cdot\oL_2\cdots\oL_n\equiv0.
\]
This proof comes largely from~\cite{yz2}, adapted here to include the base $k$ case. We can fix an open arithmetic model $\cU\to\cV$ of $X/K$ such that all $\oM,\oL_2,\dots,\oL_n\in\widehat\Pic(\cU)_{\Int}$. Recall that $\cV$ has dimension $d+1$. 

\begin{lemma}
Let $\cW$ be a proper closed subscheme of $\cV$ of dimension $e+1$, with the additional assumption when we are in the base $\Z$ case that $\cW$ is flat over $\Z$. This induces a projective and flat morphism $\cU_{\cW}\to\cW$. Then for any $\oH_1,\dots,\oH_e\in\widehat\Pic(\cV)_{\Int}$,
\[
\left(\oM|_{\cU_{\cW}}\right)^2\cdot\left(\oL_2|_{\cU_{\cW}}\right)\cdots\left(\oL_n|_{\cU_{\cW}}\right)\cdot\pi^*\left(\oH_1|_{\cW}\right)\cdots\pi^*\left(\oH_e|_{\cW}\right)=0.
\]
\end{lemma}

\begin{proof}
Assume $e=d-1$, so that $\cW\subset\cV$ has codimension $1$. The general case follows from this by induction. Our original hypothesis is that 
\[
\oM^2\cdot\oL_2\cdots\oL_n\cdot\pi^*\oH_1\cdots\pi^*\oH_d=0.
\]
for every choice of $\oH_i\in\widehat\Pic(K)_{\Int}$.
By approximation, it will suffice to prove this fixing a projective model $\cB$ for $\cV$ and assuming that every $\oH_j$ corresponds to some $\ocH_j\in\widehat\Pic(\cB)_{\Int}.$ Taking linear combinations, it suffices to further assume that $\cH_j$ is very ample. 

Let $\cC$ be the closure of $\cW$ in $\cB$. We can find a global section $s$ of $\cH_d$ which vanishes on $\cC$, and can write
\[
\div(s)=\sum_{i=0}^ra_i\cC_i
\]
where, $\cC_0=\cC$, each $\cC_i$ is an irreducible closed curve, and $a_0>0$. Now compute the intersection using this section:
\begin{multline}
0=\sum_{i=0}^ra_i\oM^2\cdot\oL_2\cdots\oL_n\cdot\pi^*\oH_1\cdots\pi^*\oH_{d-1}\cdot\pi^*\cC_i
\\-\int_{\cB(\C)}\log||s||c_1(\oM)\cdots c_1\left(\pi^*\oH_{d-1}\right).
\end{multline}
The integral on the right is not present in the base $k$ case. In base $\Z$, it simply characterizes the contribution of the Hermitian metric on $\ocH_d$. Since this is merely the intersection with a (vertical, with respect to the base $\Z$) class in $\widehat\Pic(\cB)$, it is zero by hypothesis. 

For the remaining terms,
\[
\oM^2\cdot\oL_2\cdots\oL_n\cdot\pi^*\oH_1\cdots\pi^*\oH_{d-1}\cdot\pi^*\cC_i\le0
\]
by the inequality of Theorem~\ref{hodgeindex}, and thus all must be identically zero. Since $\cC_0=\cC$ is the closure of $\cW$, this proves the lemma.

\end{proof}

We now apply the analytic theory developed in Section~\ref{analytic}. Fix $k_1$ as usual.

Let $w$ be any closed point of $\cV_{k_1}$. Then $w$ corresponds to a (flat over $\Z$, in the base $\Z$ case) dimension $1$ closed subscheme $\cW\subset\cV$. Further, $w$ determines a finite collection of closed points ($k_1$ need not be algebraically closed) of $\cV_{\nu}^{\an}$ for each place $\nu$ of $k_1$; denote this set by $w_{\nu}$. Applying the previous lemma, we have
\[
\left(\oM|_{\cU_{\cW}}\right)^2\cdot\left(\oL_2|_{\cU_{\cW}}\right)\cdots\left(\oL_n|_{\cU_{\cW}}\right)=0,
\]
and since $k_1$ is either $\Q$ or a transcendence degree one function field, we can apply the Hodge Index Theorem of~\cite{yz} and~\cite{carney} to conclude that
\[
\oM|_{\cU_{\cW}}\in\pi^*\widehat\Pic(\cW)_{\Int}
\]
whenever $X$ has good reduction at the function field of $\cW$, which holds for a dense set of points $w$. Note that in the base $k$ case, the result of~\cite{carney} additionally allows a trace component, but since $\cM$ is vertical, this is zero. 

Consider the function $\log||1||_{\oM}$ on $\cU^{\an}$. This function is constant on the fibers of $\cU^{\an}\to\cV^{\an}$ above $w_{\nu}$. Since the collection of finite sets $w_{\nu}$ over all closed points $w$ is dense in $\cV_{\nu}^{\an}$, and since $\log||1||_{\oM}$ is continuous, it is constant on all fibers of $\cU^{\an}\to\cV^{\an}$. 

Thus, $\oM\in\pi^*\widehat\Pic((\Spec K)^{\an})_{\Cont}$. To finish the proof, note that since $\oM\in\widehat\Pic(X)_{\Int}$, taking any point $x\in X(\overline K)$ we have
\[
h_{\oM}(x)\in\widehat\Pic(K)_{\Int}\quad\text{and}\quad\oM=\pi^*h_{\oM}(x).
\]




\subsection{Curves}

Next we prove Theorem~\ref{hodgeindex} for curves. Suppose $X$ has dimension $1$. In this case, our hypothesis is that $\deg M=0$ and $\oM^2\equiv0$, and our goal is to show that 
\[
\oM\in\pi^*\widehat\Pic(K)_{\Int}+\widehat\Tr_{K/k}\Pic^0(X).
\]
The idea is similar to that of Faltings~\cite{faltings} and Hriljac~\cite{hriljac}, where $\oM^2$ is related to the N\'eron-Tate height. Recall that $\widehat\Tr_{K/k}\Pic^0(X)=0$ by definition in base $\Z$. Using Lemma~\ref{flatexistence}, we can find a flat adelic line bundle $\oM_0$ such that $M_0=M$, and split $\oM$ into
\[
\oM=\oM_0+\oN,
\]
where $\oN$ is vertical. Then
\[
0\equiv\oM^2=\oM_0^2+\oN^2.
\]
Both pieces on the right are $\le0$ by the inequality part of Theorem~\ref{hodgeindex}, thus both must be $\equiv0$, and then $\oN\in\pi^*\widehat\Pic(K)_{\Int}$ by the vertical case. Thus we have reduced to proving the result for flat line bundles.

Let $\Jac_X$ be the Jacobian of $X$. After possibly extending $K$, fix a point $x_0\in X(K)$, so that we can fix a Jacobian morphism $j:X\to \Jac_X$ taking $x_0\mapsto0$. Recall from Section~\ref{flatsection} the Poincar\'e bundle $P$ on $\Jac_X\times\Jac_X$. As before, we will also write $P$ to mean it's pullback to $X\times\Jac_X$. Then $P$ can be extended to a metric $\oP$ such that $[2]^*\oP=4\oP$ on $\Jac_X\times \Jac_X$ and $(\Id\times[2])^*\oP=2\oP$ on $X\times\Jac_X$.

$M$ corresponds to a point $\alpha\in\Jac_X(K)$ and ${\oM_0=\oP|_{X\times\{\alpha\}}}$. By linearity, and possibly extending $K$, it suffices to consider the case where $\alpha=x-x_0$ for some $x\in X(K)$, so that $\alpha=j(x)$. Let $\overline x-\overline x_0$ be any lifts of the corresponding line bundles to $\widehat\Pic(X)_{\Int}$. Then
\[
\oM_0^2=\left(\oP|_{X\times\{\alpha\}}\right)^2=\overline x\cdot\oP|_{X\times\{\alpha\}}-\overline x_0\cdot\oP|_{X\times\{\alpha\}}=\pi_*\left(\oP|_{\{x\}\times\{\alpha\}}\right)-\pi_*\left(\oP|_{\{x_0\}\times\{\alpha\}}\right).
\]
The last equality holds because $\oP$ is flat, and thus the choice of metrics on $\overline x$ and $\overline x_0$ doesn't matter. 
Now, since $\overline P|_{\{x_0\}\times \Jac_X}$ is the trivial bundle, the above is just 
\[
\oM_0^2=\pi_*\left(\oP|_{\{x\}\times\{\alpha\}}\right)=h_{\oP}(x\times\alpha)=h_{\oP}(\alpha\times\alpha),
\]
Where $h_{\oP}$ is the height first on $X\times\Jac_X$, and then on $\Jac_X\times\Jac_X$.

The Poincar\'e bundle $P$ can be expressed in terms of $\Theta$ 
on $\Jac_X$, giving
\[
2\oP=\pr_1^*\overline\Theta+\pr^*\overline\Theta-m^*\overline\Theta,
\]
where $\pr_1,\pr_2,m:\Jac_X\times\Jac_X\to\Jac_X$ are the first and second projection, and addition maps, respectively~\cite{mumfordav}. We compute
\[
2h_{\oP}(\alpha\times\alpha)=h_{\overline\Theta}(\alpha)+h_{\overline\Theta}(\alpha)-h_{\overline\Theta}(2\alpha)=-2h_{\overline\Theta}(\alpha).
\]
Thus, we've proven
\[
\oM^2_0=-h_{\overline\Theta}(\alpha)=-2h_{\NT}(\alpha).
\]
We may assume $\alpha$ is not torsion, as then a multiple of $\oM$ would be vertical. Then the result follows immediately from the non-degeneracy statements for the N\'eron-Tate height, Corollaries~\ref{heightprepcorZ} and~\ref{heightprepcork}.




\subsection{Higher dimension}
Now suppose $X$ has dimension greater than one, that $M\cdot L_2\cdots L_n=0$, that $\oL_i\gg0$ for each $i$, and that
\[
\oM^2\cdot\oL_2\cdots\oL_n\equiv0.
\]
\begin{lemma}
Under the above hypoptheses, $M$ is numerically trivial.
\end{lemma}

\begin{proof}
There exists a nef adelic line bundle $\oL'$ such that $\oL_n=\oL'+\pi^*\oN$ for some $\oN\in\widehat\Pic(k_1)_{\Int}$ with $\widehat\deg(\oN)>0$. Then
\[
0\equiv\oM^2\cdot\oL_2\cdots\oL_n=\oM^2\cdot\oL_2\cdots\oL_{n-1}\cdot\oL'+(M^2\cdot L_2\cdots L_{n-1})\cdot\pi^*\oN .
\]
The first term is $\le0$ by the inequality of Theorem~\ref{hodgeindex}, and the second is $\le0$ by the inequality of the classical Hodge Index Theorem, hence both
\[
\oM^2\cdot\oL_2\cdots\oL_{n-1}\cdot\oL'\equiv0\quad\text{and}\quad M^2\cdot L_2\cdots L_{n-1}=0.
\] 
By the equality part of the classical Hodge Index Theorem, $M$ must be numerically trivial.
\end{proof}

Since $M$ is numerically trivial, it has a flat adelic extension $\oM_0$, such that $M_0=M$. As before, write $\oM=\oM_0+\oN$ where $\oN$ is vertical. Then
by the inequality,
\[
0\equiv\oM_0^2\cdot\oL_2\cdots\oL_n\equiv\oN^2\cdot\oL_2\cdots\oL_n.
\]
We've already shown that we must have $\oN\in\pi^*\widehat\Pic(K)_{\Int}$, so it remains to prove the result for $\oM_0$ flat.

We can now prove the full result by induction on the dimension of $X$. Replacing $\oL_n$ by a positive tensor power if necessary, by Seidenberg's Bertini Theorem~\cite[Theorem 7']{Seidenberg}, $L_n$ has a section which cuts out an irreducible, normal subvariety $Y\subset X$. Then
\[
0\equiv\oM_0^2\cdot\oL_2\cdots\oL_n\equiv\left(\oM_0|_Y\right)^2\cdot\oL_2|_Y\cdots\oL_{n-1}|_Y,
\]
as the two intersections differ by $\oM_0^2\cdot\oL_2\cdots\oL_{n-1}$ intersected with a vertical class, and $\oM_0$ is flat. Thus we assume by induction that 
\[
\oM_0|_Y\in\pi^*\widehat\Pic(K)_{\Int}+\widehat\Tr_{K/k}(\Pic^0(Y)).
\]
Abusing notation, we write $\pi$ to mean both structure morphisms $X\to\Spec K$ and $Y\to \Spec K$.

We proceed as in~\cite{carney}. Note that all statements about the trace below are trivially true when it is defined to be zero in the base $\Z$ setting. Write $\oM_0|_Y=\pi^*\oN+\oM_Y$, where $\oN\in\widehat\Pic(K)_{\Int}$ and $\oM_Y\in\widehat\Tr_{K/k}(\Pic^0(Y))$, and then define 
\[
\oM_X:=\oM_0-\pi^*\oN,
\]
So that $\oM_X\big|_Y=\oM_Y$.
By Lemma 3.6 of~\cite{carney}, the image of $\widehat\Tr_{K/k}\Pic^0(X)$ in $\widehat\Pic(Y)_{\Int}$ is equal to $\widehat\Tr_{K/k}(\Pic^0(Y))$ intersected with the image of $\widehat\Pic(X)_{\Int}$ via pullback of $Y\hookrightarrow X$, thus we may lift $\oM_Y$ to $\oM_X'\in\widehat\Tr_{K/k}\Pic^0(X),$ and then
\[
\oM_X-\oM_X'\in\ker\left(\widehat\Pic(X)_{\Int}\to\widehat\Pic(Y)_{\Int}\right).
\]
By the Lefshetz Hyperplane Theorem~\cite[Remark 9.5.8]{Kleiman},
$\Pic^0(X)_{\Q}\to\Pic^0(Y)_{\Q}$ is an injection, thus after replacing $\oM_0$ with a positive tensor power, we may assume $M_X-M_X'$ is trivial, so that $\oM_X-\oM_X'$ is vertical. Then by expanding and applying the Cauchy-Schwarz inequality, Lemma~\ref{cauchyschwarz}, to the mixed terms, we see that
\[
\left(\oM_X-\oM_X'\right)^2\cdot\oL_2\cdots\oL_n=\left(\oM-\pi^*\oN-\oM_X'\right)^2\cdot\oL_2\cdots\oL_n\equiv0.
\]
Thus, by the vertical case, $\oM_X-\oM_X'\in\widehat\Pic(K)_{\Int}$, and then
\[
\oM_0=\oM_X+\pi^*\oN=\left(\oM_X-\oM_X'+\pi^*\oN\right)+\oM_X',
\]
with $\oM_X-\oM_X'+\pi^*\oN\in\widehat\Pic(K)_{\Int}$ and $\oM_X'\in\widehat\Tr_{K/k}\Pic^0(X)$, finishing the proof.

\section{Arithmetic Dynamics}\label{dynamicssection}




\subsection{Admissible Metrics}
Let $f:X\to X$ be an algebraic dynamical system with a polarization $f^*L=qL$, with $L\in\Pic(X)$ ample and $q>1$. By replacing $X,f$ with their normalizations, pulling back $L$ by the normalization map, and making a finite extension of $K$ if necessary, we may assume without loss of generality that $X$ is geometrically normal. In Section~\ref{admissible} we showed that $L$ has an $f$-canonical metric and that any $M\in\Pic^0(X)$ has a flat metric. We now extend those constructions more generally. The following is much the same as in~\cite{carney,yz}, however we present the main arguments here in the current context.

First fix notation. By the N\'eron-Severi Theorem~\cite[Exp. XII, Thm 5.1, p. 650]{SGA6} and the Lang-N\'eron theorem~\cite{langneron}, the following three groups are finitely generated $\Z$-modules:
\[
\NS(X):=\Pic(X)/\Pic^0(X)\text{ (the \emph{N\'eron-Severi group})}
\]
\[
\trPo{X}:=\Pic^0(X)\bigg/\Tr_{K/k}\Pic^0(X)\cong\Pic^0(X)\bigg/\Pic^0\left(\Im_{K/k}(\Alb(X))\right)
\]
\[
\trP{X}:=\Pic(X)\bigg/\Tr_{K/k}\Pic^0(X),
\]
And we have an $f^*$-equivariant exact sequence of finite dimensional $\Q$-vector spaces
\[
0\longrightarrow\trPo{X}
_{\Q}\longrightarrow\trP{X}
_{\Q}\longrightarrow\NS(X)_{\Q}\longrightarrow0.
\]

\begin{lemma}\label{eigenvalues}
The operator $f^*$ has eigenvalues all with absolute value $q$ on $\NS(X)_{\Q}$ and eigenvalues all with absolute value $q^{1/2}$ on $\trP{X}_{\Q}$, and there exists a unique $f^*$-equivariant splitting 
\[
\ell_f:\NS(X)_{\Q}\hooklongrightarrow\trP{X}_{\Q}.
\]
\end{lemma}
\begin{proof}
The proof is nearly the same as that of the same statement, Lemma 4.1, in~\cite{carney}, which we summarize here. By the classical Hodge Index Theorem, we have the following decomposition:
\[
\NS(X)_{\R}:=\R L\oplus P(X),\quad P(X):=\left\{\xi\in\NS(X)_{\R}:\xi\cdot L^{n-1}=0\right\},
\]
and we can define a negative definite paring
\[
\langle\xi_1,\xi_2\rangle:=\xi_1\cdot\xi_2\cdot L^{n-2}
\]
on $P(X)$. Then by the projection formula for (classical) intersection numbers, $\deg f=q^n$ and $\frac1qf^*$ is orthogonal for this pairing. 

Similarly, fix any $\oH_1,\dots,\oH_d\in\widehat\Pic(K)_{\Nef}$, give $\xi_1,\xi_2\in\Pic^0(X)_{\R}$ flat metrics $\overline\xi_1,\overline\xi_2$, and give $L$ any integrable metric $\oL$. We define a pairing on $\Pic^0(X)$ by
\[
\langle\xi_1,\xi_2\rangle_{\oH_1,\dots,\oH_d}:=\overline{\xi}_1\cdot\overline{\xi}_2\cdot\overline L^{n-1}\cdot\pi^*\oH_1\cdots\pi^*\oH_d.
\]
By the definition of \emph{flat}, this does not depend on the choice of metrics. Since $\widehat\Tr_{K/k}\Pic^0(X)$ consists of flat and numerically trivial metrics, this induces a pairing ${\langle\cdot,\cdot\rangle_{\oH_1,\dots,\oH_d}}$ on $\Pic^0_{\tr}(X)_{\R}$. Then by the relative projection formula, 
\begin{multline*}
\langle f^*\xi_1,f^*\xi_2\rangle_{\oH_1,\dots,\oH_d}
=\left(f^*\xi_1\right)\cdot\left(f^*\xi_2\right)\cdot\oL^{n-1}\cdot\pi^*\oH_1\cdots\pi^*\oH_d\\
=q^n\overline\xi_1\cdot\overline\xi_2\cdot\left(\frac1qf^*\oL\right)^{n-1}\cdot\pi^*\oH_1\cdots\pi^*\oH_d
=q\langle\xi_1,\xi_2\rangle_{\oH_1,\dots,\oH_d},
\end{multline*}
making $q^{-1/2}f^*$ orthogonal for this pairing. 

We see from this that $q^{-1/2}f^*$ and $q^{-1}f^*$ are diagonalizable on $\Pic^0_{\tr}(X)_{\C}$ and $\NS(X)_{\C}$, so that 
\[
0\longrightarrow\trPo{X}_{\C}\longrightarrow\trP{X}_{\C}\longrightarrow\NS(X)_{\C}\longrightarrow0
\]
has a unique splitting as an exact sequence of $f^*$ modules. Let $Q$ and $R$ be the minimal polynomials of $f^*$ on $\Pic^0_{\tr}(X)_{\Q}$ and $\NS(X)_{\Q}$, respectively. Because the eigenvalues of $f^*$ on these two spaces have different absolute values, $Q$ and $R$ must be coprime, so that their product $QR$ is the minimal polynomial of $f^*$ on $\Pic_{\tr}(X)_{\Q}$. Define
\[
\Pic_{tr,f}(X)_{\Q}:=\ker R(f^*)|_{\Pic_{\tr}{X}_{\Q}}
\]
and then we get the desired splitting
\[
\ell_f:\NS(X)_{\Q}\isom\Pic_{tr,f}(X)_{\Q}\hooklongrightarrow\Pic_{\tr}{X}_{\Q}.
\]

\end{proof}
Adding to the above notation, define
\[
\widehat\Pic_{\tr}(X^{\an})_{\Cont}:=\widehat\Pic(X^{\an})_{\Cont}\big/\widehat\Tr_{K/k}\Pic^0(X),
\]
\[
\widehat\Pic_{\tr}(X)_{\Int}:=\widehat\Pic(X)_{\Int}\big/\widehat\Tr_{K/k}\Pic^0(X).
\]

\begin{theorem}\label{admissibletheorem}
The natural projection $\widehat\Pic_{\tr}(X)_{\Int}\to\Pic_{\tr}(X)_{\Q}$ has a unique section $L\mapsto \overline L_f$ which is equivariant under the action of $f^*$, with the following properties:
\begin{enumerate}
\item
If $L\in\Pic_{\tr}^0(X)_{\Q}$ then $\overline L_f$ is flat.
\item
If $L\in\Pic_{\tr}(X)_{\Q}$ is ample then $\overline L_f$ is nef.
\end{enumerate}
We call such a metric \emph{$f$-admissible}.
\end{theorem}

\begin{remark}\label{admissiblemodtrace}
Since $\widehat\Tr_{K/k}\Pic^0(X)$ is flat and numerically trivial, and since the underlying line bundles are also numerically trivial in $\Pic(X)$, flatness, nefness, and amplitude are all still well defined notions modulo the trace. Abusing notation, we will often write $\overline L_f$ to mean any lift of $\overline L_f$ from $\widehat\Pic_{\tr}(X)_{\Int}$ to $\widehat\Pic(X)_{\Int}$, as heights and intersections involving $\overline L_f$ will not depend on the choice of lift.
\end{remark}

\begin{proof}
We first produce a section of the projection $\widehat\Pic_{\tr}(X^{\an})_{\Cont}\to\Pic_{\tr}(X)_{\Q}$, and then show that both properties hold for this section. Since $\Pic_{\tr}(X)_{\Q}$ is generated by $\Pic_{\tr}^0(X)_{\Q}$ and the ample elements of $\Pic_{\tr,f}(X)_{\Q}$, and flat and nef metrics are integrable, this proves the theorem.

We have an exact sequence
\[
0\longrightarrow
C(X^{\an})\longrightarrow
\widehat\Pic_{\tr}(X^{\an})_{\Cont}
\longrightarrow
\trP{X}_{\Q}\longrightarrow
0,
\]
Where $C(X^{\an})$ is the group of continuous metrics on $\cO_X$. Equivalently, $C(X^{\an})=\bigoplus_{\nu}C(X^{\an}_{\nu})$, where $C(X^{\an}_{\nu})$ is the group of continuous $\R$-valued functions on $X^{\an}_{\nu}$ via the map $||\cdot||\mapsto-\log||1||_{\nu}$. The supremum norm turns this into a Banach space.

Recall the minimal polynomial $QR$ for the operator $f^*$ on $\trP{X}_{\Q}$, as defined in Lemma~\ref{eigenvalues}. 
\begin{lemma}
$QR(f^*)$ is invertible on $C(X^{\an})$
\end{lemma}
\begin{proof}
Over $\C$, we can factor $QR(T)$ as
\[
QR(T)=a\prod_i\left(1-\frac{T}{\lambda_i}\right),
\]
where $a\ne0$, and by lemma \ref{eigenvalues}, $|\lambda_i|$ is either $q^{\frac12}$ or $q$, both of which are greater than $1$. Each term of this product has inverse
\[
\left(1-\frac{f^*}{\lambda_i}\right)^{-1}=\sum_{j=0}^{\infty}\left(\frac{f^*}{\lambda_i}\right)^j,
\]
provided this series converges absolutely with respect to the operator norm, which is defined with respect to the supremum norm. 
But $f^*$ does not change the supremum norm, so the operator norm of $f^*$ is $1$, and 
\[
\left\lVert\left(\frac{f^*}{\lambda_i}\right)^j\right\lVert_{op}=\frac1{|\lambda_i|^j}\le q^{-\frac j2},
\]
so the series converges absolutely.
\end{proof}

Since $QR(f^*)$ is invertible on the image of $C(X^{\an})$ in $\widehat\Pic_{\tr}(X^{\an})_{\Cont}$ and annihilates $\trP{X}_{\Q}$, we have an $f^*$-invariant decomposition
\[
\widehat\Pic_{\tr}(X^{\an})_{\Cont}=C(X^{\an})\bigoplus\ker(QR(f^*)).
\]
Let $\overline M$ consist of $M\in\trP{X}_{\Q}$ with any choice of metric, modulo $\widehat\Tr_{K/k}\Pic^0(X)$. Then a well-defined splitting is given by
\[
\overline M_f:=\overline M-\left(QR(f^*)|_{C(X^{\an})}\right)^{-1}\circ QR(f^*)\overline M.
\]

We now show that this splitting satisfies properties (1) and (2). For (1), suppose $L\in\Pic_{\tr}^0(X)_{\Q}$. By the definition and universal property of the Albanese morphism, $f$ induces a map $f':\Alb_X\to\Alb_X$, and $L$ corresponds to a line bundle $L_0\in\Pic_{\tr}^0(\Alb_X)$ such that $\overline L_f=\iota^*\overline L_{0,f'}$. Since $[2]$ and $f'$ commute on $\Alb_X$, we also have 
\[
[2]^*\overline L_{0,f'}=\overline{\left([2]^*L_0\right)}_{f'}= 2\overline L_{0,f'}.
\] By the construction in the proof of Lemma~\ref{flatexistence}, this makes $\overline L_{0,f'}$ flat, and the pullback of a flat adelic line bundle is again flat.

For (2), let $L$ be ample, let $\overline L_0$ be any nef extension of $L$, and define a sequence
\[
\overline L_m:=q^{-1}f^*\overline L_{m-1},
\]
which consists of all nef adelic line bundles. \cite[Theorem 4.9]{yz} proves that this has a subsequence converging to $\overline L_f$. Their proof assumes $K$ is a number field, but only requires the fact that $\Pic_{\tr,f}(X)_{\Q}$ is a finite dimensional $\Q$-vector space on which the operator $q^{-1}f^*$ has eigenvalues with absolute value one.
\end{proof}

Combining this construction with that of Lemma~\ref{eigenvalues}, we have an $f$-equivariant map
\[
\widehat\ell_f:\NS(X)_{\Q}\to\widehat\Pic_{\tr}(X)_{\Int}.
\]
As noted in Remark~\ref{admissiblemodtrace}, we will often treat $\widehat\ell_f$ as having image in $\widehat\Pic(X)_{\Int}$, as the choice of lift from the quotient modulo the trace will not matter when computing heights and intersections.




\subsection{Rigidity of dynamical systems}

We now turn to proving Theorem~\ref{dynamicstheorem}. Recall that $X$ is a projective variety over any field $\K$, and $f$ and $g$ are two polarizable endomorphisms of $X$. We fix polarizations $f^*L=qL$ and $g^*M=rM$, where $L$ and $M$ are ample and $q,r>1$. $X$ is defined by finitely many homogeneous polynomials, $f$ and $g$ are specified by finitely many polynomials, and $L$ and $M$ are determined by \v{C}ech cocyles in $H^1(X,\cO_X^*)$, which are in turn determined by finitely many polynomials. The polarizations, as morphisms of such, are also given by finitely many polynomials. Thus, by the Lefschetz principle, 
we may assume that all these data are defined over a subfield $K\subset\K$ which is finitely generated over its prime field, either $\Q$ or $\F_p$. Let $k$ be this prime field. Recall the conditions we aim to show are equivalent.
\begin{enumerate}
\item 
$g(\Prep(f))\subset\Prep(f)$.
\item
$\Prep(f)\subset\Prep(g)$.
\item
$\Prep(f)\cap\Prep(g)$ is Zariski dense in $X$.
\item
$\Prep(f)=\Prep(g)$.
\end{enumerate}

To start, that $(4)$ implies $(1)$ is clear. Now assume $(1)$ holds, and stratify $\Prep(f)$ by degree:
\[
\Prep(f)=\bigcup_{d\ge1}\{x\in\Prep(f):[K(x):K]\le d\}.
\]
Since all of $\Prep(f)$ has $f$-canonical height zero, each set on the left has bounded height and bounded degree, and is thus finite by Lemmas~\ref{heightprepcorZ} and \ref{heightprepcork}. Since $g$ is also defined over $K$, it fixes this stratification, and every element of $\Prep(f)$ has finite forward orbit under $g$, proving $(2)$. 

Fakhruddin~\cite{Fakhruddin} shows that, over any field, $\Prep(f)$ is dense in $X$, and then $(2)$ implies $(3)$ immediately. 

Since $k$ is $\Q$ or finite, the sets of $f$-preperiodic points and $f$-canonical height zero points are equal, and similarly for $g$. Thus the following theorem, whose proof takes up the remainder of this section, shows that $(3)$ implies $(4)$, and completes the proof of Theorem~\ref{dynamicstheorem}. 

\begin{theorem}\label{heightsdynamictheorem}
Let $k$ be any field, let $K$ be a finitely generated extension of $k$, let $X$ be a projective variety over $K$, and let $f$ and $g$ define two polarizable dynamical systems on $X$. Let $M$ be any ample $\Q$-line bundle on $X$. Then the set
\[
Z_{f}:=\{x\in X(\overline K)|h_{\oM_f}(x)\equiv0\},
\]
does not depend on the choice of $M$, and the following are equivalent.
\begin{enumerate}
\item
$Z_f\cap Z_g$ is Zariski dense in $X$.
\item
$\widehat\ell_f\equiv\widehat\ell_g$.
\item
$Z_f=Z_g$
\end{enumerate}
\end{theorem}

\begin{remark}
In fact Theorem~\ref{heightsdynamictheorem} is sufficient to prove Theorem~\ref{dynamicstheorem} even without invoking the Lefschetz principle, provided $X$ is totally non-isotrivial over $k$. In effect any coarser height functions on $X$ defined relative to larger fields $k$ are sufficient to demonstrate this rigidity result, provided they are not so coarse as to be trivial on some sub-dynamical system. 
\end{remark}

\begin{proof}
First, to see that $Z_f$ doesn't depend on $M$, suppose $M,M'\in\Pic(X)$ are both ample. Fix any $\oH_1,\dots,\oH_d\in\widehat\Pic(K)_{\Nef}$. Then $\oM_f$ and $\oM'_f$ are both nef, and there exist real constants $C,C'>0$ such that 
\[
0\le h_{\oM_f}^{\oH_1,\dots,\oH_d}(x)\le C\cdot h_{\oM_f'}^{\oH_1,\dots,\oH_d}(x)\le C'\cdot h_{\oM_f}^{\oH_1,\dots,\oH_d}(x)
\]
for all $x\in X(\overline K)$.

Now we prove that $(1)$ implies $(2)$, as $(2)$ implies $(3)$ and $(3)$ implies $(1)$ are clear. Fix an ample class $\xi\in\NS(X)$. Then $\oL_f:=\widehat\ell_f(\xi)$ and $\oM_g:=\widehat\ell_g(\xi)$ are both nef by Theorem~\ref{admissibletheorem} and so is their sum, $\oN:=\oL_f+\oM_g$. Since $Z_f\cap Z_g$ is dense in $X(\overline K)$, we have 
\[
h_{\oN}(x)=h_{\oL_f}(x)+h_{\oM_g}(x)\equiv0
\]
on a dense subset of $X(\overline K)$. Now fix any $\oH\in\widehat\Pic(K)_{\Int}$ satisfying the Moriwaki condition, Definition~\ref{moriwakicondition}, and by Theorem~\ref{essentialminimum},
\[
0=\lambda_1^{\oH}(X,\oN)\ge h^{\oH}_{\oN}(X)\ge0. 
\]
Thus,
\[
\oN^{n+1}\cdot\oH^d=(\oL_f+\oM_g)^{n+1}\cdot\oH^d=0
\]
holds for any $\oH$ satisfying the Moriwaki condition, and then by Lemma~\ref{moriwakinumericallytrivial}, $(\oL_f+\oM_g)^{n+1}\equiv0$. Expand this product. Since $\oL_f$ and $\oM_g$ are both nef, all terms must be numerically trivial, and then in particular,
\[
\left(\oL_f-\oM_g\right)^2\cdot\left(\oL_f+\oM_g\right)^{n-1}\equiv0.
\]
Fix a subfield $k_1$ as usual, and let $\overline Q\in\widehat\Pic(k_1)_{\Int}$ have positive degree. Then $\oL_f+\oM_g+\pi^*\overline Q\gg0$, and since $L_f-M_g$ is numerically trivial, making $\oL_f-\oM_g$ flat,
\[
\left(\oL_f-\oM_g\right)^2\cdot\left(\oL_f+\oM_g+\pi^*\overline Q\right)^{n-1}\equiv0\quad\text{and}\quad \left(L_f-M_g\right)\cdot\left(L_f+M_g\right)^{n-1}=0,
\] 
Thus by Theorem~\ref{hodgeindex},
\[
\oL_f-\oM_g\in\pi^*\widehat\Pic(K)_{\Int}+\widehat\Tr_{K/k}\Pic^0(X).
\]
Finally, since elements of $\widehat\Tr_{K/k}\Pic^0(X)$ produce heights which are identically zero, and elements of $\pi^*\widehat\Pic(K)_{\Int}$ produce constant heights, we can conclude from ${Z_f\cap Z_g\ne\emptyset}$ that $\widehat\ell_f\equiv\widehat\ell_g$, that $h_{\oL_f}=h_{\oM_g}$, and that $Z_f=Z_g.$
\end{proof}

\subsection{Equality of equilibrium measures}
To finish, we prove Corollary~\ref{introequilibrium}. Let $\K$ be algebraically closed and complete with respect to some absolute value, and let $f:X\to X$ be a polarizable dynamical system over $\K$. Let $X^{\Ber}$ denote the Berkovich analytic space associated to $X$; note that since $\K$ is a complete field this is the usual Berkovich space and does not require the generalization from Section~\ref{analytic}. Then the \emph{equilibrium measure}, $d\mu_f$ is the unique smooth measure on $X^{\Ber}$ satisfying $f^*d\mu_f=q^{\dim X}d\mu_f$ and $f_*d\mu_f=d\mu_f$. 

The existence of $d\mu_f$ is proven by starting with any smooth measure on $X^{\Ber}$ and using Tate's limiting argument as in Section~\ref{tateslimiting}. We now prove the corollary.

\begin{proof}
As before, by the Lefschetz principle, we can find a subfield $K$ of $\K$ over which $X,f,g,$ and the requisite line bundles are defined, and which is finitely generated over the prime field. Let $\eta$ be the natural inclusion $\eta:K\to\K$. The absolute value on $\K$, composed with $\eta$, induces a point $\eta^{\an}\in\left(\Spec K\right)^{\an}$, and the fiber $X^{\an}_{\eta}$ of $X^{\an}$ (here we do use the analytic space defined in Section~\ref{analytic}) over $\eta^{\an}$ is simply $X^{\Ber}$.

Define $\widehat\Pic(X)_{\Q}$ to be the group of $\Q$-line bundles on $X$ with continuous $\K$-metrics on $X^{\Ber}$. Restricting the above equality $\widehat\ell_f\equiv\widehat\ell_g$ over $K$ to the fiber $X^{\Ber}$, we get an induced equality $\widehat\ell_{f,\K}\equiv\widehat\ell_{g,\K}$ of sections
\[
\widehat\ell_{f,\K},\widehat\ell_{g,\K}:\NS(X)_{\Q}\longrightarrow\widehat\Pic(X)_{\Q}.
\]
Now let $\xi\in \NS(X)_{\Q}$ be any ample class. Then $\ell_f(\xi)$ can be decomposed into $f^*$ eigencomponents $L_1,\dots, L_n\in\Pic(X)_{\C}$, and each can be lifted by $\widehat\ell_{f,\K}$ to $f^*$ eigenvectors $\oL_{1,f},\dots,\oL_{n,f}$ on $\widehat\Pic(X)_{\C}$. Using the arithmetic Chern classes defined in~\cite{ChambertLoirDiophantine}, since both sides are uniquely $f^*$-invariant, we have
\[
d\mu_f=\frac1{L_1\cdots L_n}\widehat c_1(\oL_{1,f})\wedge\dots\wedge \widehat c_1(\oL_{n,f})
\]
as linear functionals on $X^{\Ber}$. Thus $d\mu_f=d\mu_g$.
\end{proof}


\end{document}